\documentclass[11pt]{amsart}
\usepackage[T1]{fontenc}
\usepackage[utf8]{inputenc}
\usepackage{lmodern}
\usepackage{amsmath,amssymb,amsthm,mathtools}
\usepackage{enumitem}

\numberwithin{equation}{section}

\newtheorem{theorem}{Theorem}[section]
\newtheorem{lemma}[theorem]{Lemma}
\newtheorem{proposition}[theorem]{Proposition}
\newtheorem{corollary}[theorem]{Corollary}
\theoremstyle{definition}
\newtheorem{example}[theorem]{Example}
\newtheorem{definition}[theorem]{Definition}
\newtheorem{question}[theorem]{Question}
\theoremstyle{remark}
\newtheorem{remark}[theorem]{Remark}

\DeclareMathOperator{\conv}{conv}
\DeclareMathOperator{\codim}{codim}
\DeclareMathOperator{\Ker}{Ker}
\DeclareMathOperator{\Int}{Int}
\DeclareMathOperator{\supp}{supp}
\DeclareMathOperator{\dist}{dist}

\newcommand{\R}{\mathbb{R}}
\newcommand{\C}{\mathbb{C}}
\newcommand{\N}{\mathbb{N}}
\newcommand{\T}{\mathbb{T}}
\newcommand{\D}{\mathbb{D}}
\def\cA{\mathcal A}
\newcommand{\ip}[2]{\langle #1,#2\rangle}
\newcommand{\e}{\varepsilon}
\newcommand{\We}{W_e}
\newcommand{\wto}{\xrightarrow{\mathrm w}}
\newcommand{\wstarto}{\xrightarrow{\mathrm w^*}}

\begin{document}

\title[Numerical and essential numerical ranges on $\ell_p$]
{Numerical and essential \\numerical ranges on $\ell_p$}
\author[V. M\"uller]{Vladimir M\"uller}
\address{Institute of Mathematics, Czech Academy of Sciences, \v{Z}itn\'a 25, 115 67 Prague, Czech Republic}
\email{muller@math.cas.cz}

\author[Y. Tomilov]{Yuri Tomilov}
\address{Institute of Mathematics, Polish Academy of Sciences, \'Sniadeckich 8, 00-656 Warsaw, Poland}
\address{Faculty of Mathematics and Computer Science, Nicolaus Copernicus University, Chopina 12/18, 87-100 Toru\'n, Poland}
\email{ytomilov@impan.pl}

\thanks{The research was supported by GA CR/NCN grant 25-15444K. The first author was supported by the Czech Academy of Sciences (RVO:67985840). The second author was supported by the NCN grant Weave-Unisono, 2024/06/Y/ST1/00044. The second author was also partially supported by the NCN grant UMO-2023/49/B/ST1/01961 and the NAWA/NSF grant BPN/NSF/2023/1/00001.}
\date{}
\begin{abstract}
The paper offers the first systematic study of ordinary and essential numerical ranges of operators on $\ell_p$, $1<p<\infty$, as an atomic picture within a broader $L^p$ project. The paper begins with Banach-space foundations, including the finite-codimensional description of the essential numerical range and a Banach-space convex-hull inclusion for the essential spectrum. It then turns to finite-dimensional $\ell_p$ geometry, where one finds both positive star-shapedness phenomena and explicit $2\times2$ counterexamples. On $\ell_p$, we prove that the essential numerical range is compact and convex, identify it with the algebraic numerical range of the Calkin image, obtain a compact-perturbation formula, and show that, moreover, the closure of the numerical range is star-shaped, while points in the interior of the essential numerical range are exact star-centres of the numerical range.
The paper illustrates the developed theory with sequence-space examples, covering tridiagonal Toeplitz operators and the discrete Hilbert transform, and, after relating our study to a variant of the Crouzeix inequality, closes with a brief discussion of extensions to spaces of class $(P)$ and to joint essential numerical ranges.
\end{abstract}

\maketitle

\tableofcontents

\section*{Introduction}
The numerical range $W(T)$ of a linear operator $T$ on a Hilbert space $H$, and its generalisations, are of fundamental importance throughout operator theory, including spectral theory, functional calculi, dilation theory and operator-norm estimates. 
A good contemporary account of most of the important features of $W(T)$ can be found in \cite{GW}, and recent papers \cite{AglerLykovaYoung2024}, \cite{Bickel2020}, \cite{BHSVW2025},  \cite{Bogli2020}, \cite{LauLiPoonSze2018}
and \cite{Malman2025} can serve as pertinent illustrations.
 The applications of numerical ranges also extend to PDE theory, where they are used as spectral enclosures, to control theory and engineering, where they appear in the study of structured eigenvalue problems and convex and non-convex optimisation, and to mathematical physics and quantum information theory. A fundamental property of $W(T)$ is its convexity, discovered by Toeplitz and Hausdorff. This observation led to a surprisingly fruitful theory, to many useful generalisations, such as higher-rank, infinite, essential and joint numerical ranges, and to convexity phenomena in related contexts, see, in particular, \cite{GW} and \cite{GR}. At the same time, topological properties of $W(T)$ remain rather mysterious and are still not well understood. The situation is even more difficult on non-Hilbert spaces, where very little information is available.

Outside Hilbert space the situation changes drastically: the spatial numerical range is typically non-convex, duality becomes less rigid, and even basic geometric conclusions require additional arguments.
This is already visible on $\ell_p$ for $p\neq2$, see Section~\ref{secfinite}.

Numerical ranges on non-Hilbert spaces have remained a rather obscure subject, and only very limited results seem to be available in this direction. One reason is that the study of numerical ranges on such spaces is technically demanding even in low-dimensional situations: already the case of rank-two operators presents substantial difficulties. This stands in sharp contrast with the Hilbert space setting, where many basic structural facts about numerical ranges can be reduced to the analysis of numerical ranges of $2\times2$ matrices.

The purpose of the present paper is to isolate the atomic $\ell_p$ part of a broader $L^p$ project. We adopt the view that the atomic case should be treated first. It is rich enough to display the main non-Hilbertian phenomena, but still concrete enough to allow explicit calculations, exact examples, and a reasonably complete atomic picture. The non-atomic $L^p$ theory, with its genuinely different mechanisms, is better treated separately.

Since the paper is concerned mainly with $\ell_p:=\ell_p(\mathbb N)$, we record already here how the two numerical ranges should be read in that setting. Throughout this introductory discussion $1<p<\infty$ and $q$ is the conjugate exponent. We identify $(\ell_p)^*$ with $\ell_q$ by means of the pairing
\[
        \ip{x}{y}
        =\sum_{j\ge1}x_j\overline{y_j}, \qquad 
        x=(x_j) \in\ell_p,\qquad y=(y_j)\in\ell_q .
\]
For $x=(x_j)$ on the unit sphere of $\ell_p$, the unique norming functional is represented by
\[
        J(x)=(x_j|x_j|^{p-2})_{j\ge1}\in\ell_q,
        \qquad
        \|J(x)\|_q=1=\ip{x}{J(x)}.
\]
Thus $W(T)$, for $T\in\mathcal L(\ell_p)$, is obtained by testing $T$ on pairs $(x,J(x))$ with $x$ of norm one:
\[
        W(T)=\bigl\{\ip{Tx}{J(x)}:\ \|x\|_p=1\bigr\}.
\]
In the same concrete $\ell_p$ setting, the essential numerical range may be described sequentially: a complex number $\lambda$ belongs to $W_e(T)$ precisely when there are vectors $x_n\in\ell_p$ such that
\[
        \|x_n\|_p=1,
        \qquad x_n\wto0,
        \qquad \ip{Tx_n}{J(x_n)}\to\lambda .
\]
The formal Banach-space definitions, including the general net version, are recalled in Section~\ref{prelim}.

In this work, we analyse a number of examples showing that the ordinary numerical range on \(\ell_p\), \(p\neq 2\), retains several features familiar from its Hilbert-space counterpart, corresponding to \(p=2\). At the same time, some of these examples reveal substantial differences between the Hilbertian and non-Hilbertian cases. In particular, for \(p\neq 2\), convexity already fails in dimension two.
 Against this background it is natural to ask whether the essential
numerical range retains any of the Hilbert-space features. We show that, on
\(\ell_p\), it does: \(W_e(T)\) is convex, coincides with the algebraic numerical
range of the Calkin image \(\pi(T)\), contains \(\operatorname{conv}\sigma_e(T)\),
and has the expected compact-perturbation and star-shapedness consequences,
including its relation to $W(T).$
The ordinary numerical range on \(\ell_p\), \(p\ne2\), appears to be genuinely non-Hilbertian. Nevertheless, it retains some vestiges of convexity. In the infinite-dimensional \(\ell_p\) setting, \(W(T)\) is governed to a considerable extent by \(W_e(T)\), and \(\overline{W(T)}\) is star-shaped with every point of \(W_e(T)\) as a star-centre. It remains open whether \(W(T)\) itself is always star-shaped, even for operators on finite-dimensional \(\ell_p\), \(p\ne 2\).



This is close in spirit to the philosophy of \(L^p\)-operator algebras, where one
asks which Hilbert-space phenomena survive after replacing Hilbert spaces by
\(L^p\)-spaces, see e.g. \cite{Blecher2019,Boedihardjo2024}. Here the same question is pursued for
numerical ranges, with the essential numerical range playing the role of the
stable, Calkin-level object.

Along the way we also obtain several results on numerical ranges on general Banach spaces that may be of independent interest, and we identify large natural subsets of numerical ranges on $\ell_p$. These sets again point to traces of convexity in the geometry of general $\ell_p$-spaces.

A non-atomic counterpart is developed separately in~\cite{MTnonatomic}. Some weaker convexity phenomena survive there as well, but the arguments are considerably more delicate and rely on methods that have no atomic analogue, most notably disjoint-support and shrinking-support constructions on non-atomic measure spaces.

The paper is organised as follows. Section~\ref{prelim} motivates and elaborates the
background used throughout the paper. After a short discussion of
Hilbert space numerical ranges, we recall the ordinary and essential numerical
ranges on Banach spaces, obtain a new characterisation of
\(W_e(T)\) in terms of approximately biorthogonal nets, and
prove the Banach-space spectral inclusion
\[
        \conv\sigma_e(T)\subset W_e(T)
\]
for the essential spectrum $\sigma_e(T).$

Section~\ref{secfinite} turns to finite-dimensional \(\ell_p\)-spaces. The
main purpose of this section is to isolate the rigid \(2\times2\) geometry that
underlies the later atomic theory. We obtain an explicit parametrisation of the
numerical range, derive a non-convexity example which extends to
every finite-dimensional \(\ell_p\)-space of dimension at least two, prove that the convex hull of the diagonal entries
is always contained in the numerical range, establish star-shapedness for
triangular \(2\times2\) matrices, and present an explicit counterexample
showing that the trace centre need not be a star-centre for general
\(2\times2\) matrices.

Section~\ref{inflp} contains the atomic $\ell_p$ core of the
paper. After recalling the continuity properties of the duality map \(J\), we
show that the non-convex finite-dimensional geometry persists on \(\ell_p\).
We then prove that the numerical range contains all countable convex
combinations of the diagonal values \(\ip{Te_j}{J(e_j)}\), with respect to the
standard unit basis \((e_j)_{j\ge1}\) in \(\ell_p\). For the essential numerical
range we obtain a tail-space characterisation, a finite-codimensional
annihilator form, convexity, a Calkin-algebra description, and the
compact-perturbation formula. For
\(\lambda\in\Int W_e(T)\), we obtain a \(\lambda\)-diagonal compression statement,
thus matching its Hilbert space counterpart.
Finally, we prove that
\[
        \Int W_e(T)\subset W(T),
\]
that \(\overline{W(T)}\) is star-shaped with every point of \(W_e(T)\) as a
star-centre, and that, when \(\Int W_e(T)\ne\varnothing\), every point of
\(\Int W_e(T)\) is a star-centre of \(W(T)\).

Section~\ref{secexamp} collects sequence-space examples that still belong
naturally to the atomic side of the theory. We treat unilateral Toeplitz
operators and explicit tridiagonal examples on \(\ell_p(\mathbb N_0)\), the
discrete Hilbert transform on \(\ell_p(\mathbb Z)\), and the unilateral shift.
The section ends with a discussion of the failure of Crouzeix-type estimates
on \(\ell_p\) and with finite-dimensional shift examples leading to a natural
question about the correct order of finite-dimensional Crouzeix-type bounds.

Section~\ref{secexten} indicates two natural extensions of the present work:
spaces of class \((P)\) and joint essential numerical ranges. The appendix
contains the computational proof of the \(\ell_3^2\) counterexample used in
Section~\ref{secfinite}.

To keep the paper within reasonable limits, we do not elaborate here
numerical ranges on \(\ell_p\) for special classes of operators, such as
essentially Hermitian operators and related lifting questions, cf.
\cite{AllenWard79}, nor several further structural questions.

\section{Numerical and essential numerical ranges on Banach spaces}\label{prelim}
\subsection{Notation and conventions}

We collect the notation used throughout the paper. We write
\[
\D:=\{z\in\C: |z|<1\},
\qquad
\T:=\{z\in\C: |z|=1\}.
\]
Thus $\overline\D$ denotes the closed unit disc. For a subset $E\subset\C$ we write $\overline E$ for its closure, $\partial E$ for its boundary, $\Int E$ for its interior, and $\conv E$ for its convex hull.

If $X$ is a Banach space, then $\mathcal L(X)$ denotes the algebra of all bounded linear operators on $X$, and $X^*$ denotes the dual space. The duality pairing is written as $\ip{x}{x^*}$. For a subset $A\subset X$ we put
\[
A^\perp:=\{x^*\in X^*: \ip{a}{x^*}=0 \text{ for all } a\in A\}.
\]
We write $x_\alpha\wto x$ for weak convergence in $X$, and $x_\alpha^*\wstarto x^*$ for weak-* convergence in $X^*$. We also set
\[
\Pi(X):=\{(x,x^*)\in X\times X^*: \|x\|=\|x^*\|=1=\ip{x}{x^*}\}.
\]
Throughout, finite-codimensional subspaces are understood to be closed.
In expressions like ``$x_n \to x$ as $n \to \infty$'' we omit the range of the parameter
whenever the nature of the limit is clear.

If $T\in \mathcal L(X)$, then $\sigma(T)$ denotes its spectrum. We write $\mathcal K(X)$ for the ideal of compact operators on $X$.

Throughout the paper, \(\ell_p\), \(1<p<\infty\), stands for the usual complex sequence space
\(\ell_p(\mathbb N)\).
When the underlying index set is relevant, we write it explicitly, for example
\(\ell_p(\mathbb N_0)\) or \(\ell_p(\mathbb Z)\), where
\(\mathbb N_0=\mathbb N\cup\{0\}\). We denote by \(\ell_p^n\) the space
\(\mathbb C^n\) equipped with the canonical \(\|\cdot\|_p\)-norm.

On these sequence spaces we use the convention 
\[
        \ip{x}{y}=\sum_{j} x_j\overline{y_j},
        \qquad x=(x_j) \in\ell_p,\quad y=(y_j) \in\ell_q,
\]
under the standard identification \((\ell_p)^*=\ell_q\), where \(q\) is the conjugate exponent.

\subsection{Hilbert-space theory}

For orientation and to put our results into proper context, we recall the classical Hilbert-space picture.
It will be a pattern for the rest of the paper.
Let $H$ be a complex Hilbert space and let $T\in \mathcal L(H)$. Then
\[
W(T):=\{\ip{Tx}{x}: x\in H,\ \|x\|=1\}
\]
denotes the numerical range of $T$, and
\[
w(T):=\sup\{|\lambda|: \lambda\in W(T)\}
\]
denotes the numerical radius.

The numerical range is a fundamental concept in Hilbert-space operator
theory, with numerous applications. A concise account of its basic properties
can be found in \cite{GR}. For a more recent and comprehensive treatment,
see \cite{GW}. We now recall the facts from the Hilbert-space theory that
will serve as a reference point for the rest of the paper.


\begin{theorem}
Let $T\in \mathcal L(H)$. Then:
\begin{itemize}
\item[(i)] $W(T)$ is a convex subset of $\C$;
\item[(ii)] $\frac{\|T\|}{2}\le w(T)\le\|T\|$;
\item[(iii)] $\sigma(T)\subset\overline{W(T)}$;
\item[(iv)] $T$ is selfadjoint if and only if $W(T)\subset\R$. In this case, $\overline{W(T)}=\conv\sigma(T)$ and $w(T)=\|T\|$.
\end{itemize}
\end{theorem}

There are many variants of the numerical range of Hilbert space operators
which also exhibit convexity, see, for instance, \cite[Sections~4 and~8]{GW}
and also \cite[Section~5]{MT-jointnr}, \cite{MT} and the references therein.
The most important of these variants for the present paper, and also for
applications, is the essential numerical range. In the Hilbert-space setting
it was systematically studied in the foundational paper \cite{FSW}.

On a Hilbert space $H$, the essential numerical range of $T\in \mathcal L(H)$ is defined by
\[
\begin{split}
W_e(T):=\{\lambda\in\C: &\ \exists\ (x_n)\subset H,\ \|x_n\|=1\ \text{for all }n,\\
&\ x_n\wto0,\ \ip{Tx_n}{x_n}\to\lambda\}.
\end{split}
\]
Equivalently, one may require only that 
$(x_n)\subset H$ is an orthonormal sequence 
 and $\ip{Tx_n}{x_n}\to\lambda$.

We also use the corresponding essential spectrum,  which is easy to describe in terms of the Calkin algebra $\mathcal L(H)/\mathcal K(H).$ Let $\mathcal K(H)$ denote the ideal of all compact operators on $H$, and let
\[
\pi:\mathcal L(H)\to \mathcal L(H)/\mathcal K(H)
\]
be the canonical quotient map. For $T\in \mathcal L(H)$ we then define
\[
\sigma_e(T):=\sigma\bigl(\pi(T),\mathcal L(H)/\mathcal K(H)\bigr).
\]

The basic Hilbert-space properties of the essential numerical range may be summarised as follows. 

\begin{theorem}\label{hilbertwe}
Let $T\in \mathcal L(H)$.
Then:
\begin{enumerate}[label=\textup{(\roman*)}]
\item $W_e(T)$ is a compact convex subset of $\C$;
\item $\sigma_e(T)\subset W_e(T)$;
\item $W_e(T)=\bigcap\{\overline{W(T+L)}: L\in \mathcal K(H)\};$
\item For \(\lambda\in\C\), one has $\lambda\in W_e(T)$ if and only if for every subspace $M\subset H$ of finite codimension and every $\e>0$ there exists a unit vector $x\in M$ such that $|\langle Tx,x\rangle-\lambda|<\e$;
\item If \(\lambda\in\operatorname{Int} W_e(T)\), then \(\lambda\in W(T)\). Moreover, for every subspace $M\subset H$ of finite codimension there exists a unit vector $x\in M$ such that $\langle Tx,x\rangle=\lambda$.

\end{enumerate}
\end{theorem}
The properties (i)--(iii) are standard and can be found, for instance, in
\cite[Chapter~4.2]{GW}. For the proof of (iv), see e.g.
\cite[Lemma~4.1]{MT-circles}, while the proof of (v) can be found in
\cite[Corollary~4.2]{MT-jointnr}. The latter two properties are established
there in the more general setting of operator tuples.

Both the numerical range and the essential numerical range are intimately related to the notion of algebraic numerical range, and in particular to the numerical range in Calkin algebras. Recall that if $\cA$ is a unital Banach algebra with unit $1_\cA$ and $a\in\cA$ then the algebraic numerical range $V(a,\cA)$ of $a$ is given by
$$
V(a,\cA)=\{f(a):f\in\cA^*, \|f\|=1=f(1_\cA)\},
$$
where $\cA^*$ denotes the dual of $\cA$, i.e., the space of all bounded linear functionals on $\cA$. The algebraic numerical range in any unital Banach algebra is a closed convex set, see e.g. \cite[Chapter~1.2]{BD1} or \cite[Chapter~4.1]{GW}.

The next statement records the Hilbert-space identities for the algebraic
numerical range, including the Calkin-algebra identity generalised below, see e.g.
\cite[Chapters~4.1 and 4.2]{GW}. More general theory is discussed in
\cite[Chapter~3, \S9]{BD1} and \cite[Chapter~7, \S34]{BD2}, see also
Theorem~\ref{vrange} below. 

\begin{theorem}\label{vrangeh}
Let $T\in \mathcal L(H)$. Then:
\begin{enumerate}[label=\textup{(\roman*)}]
\item $V(T,\mathcal L(H))=\overline{W(T)}$;
\item  $V\bigl(\pi(T),\mathcal L(H)/\mathcal K(H)\bigr)=W_e(T),$
where \(\pi:\mathcal L(H)\to\mathcal L(H)/\mathcal K(H)\) is the quotient map.
\end{enumerate}
\end{theorem}

The rest of the paper studies, in particular, which parts of this Hilbertian picture survive, and in what form, on $\ell_p$ for $p\neq2$.

\subsection{Numerical ranges in Banach spaces}\label{Banachtheory}

Given a complex Banach space $X,$ the numerical range of an operator $T\in \mathcal L(X)$ is defined by
\begin{equation}\label{eq:banach-spatial-numerical-range}
W(T):=\{\ip{Tx}{x^*}: (x,x^*)\in\Pi(X)\},
\end{equation}
where $\Pi(X)$ is as in the notation above.
Then its numerical radius is given by
\[
w(T):=\sup\{|\lambda|: \lambda\in W(T)\}.
\]

Standard references for the general theory of numerical ranges include
\cite{BD1} and \cite{BD2}. The numerical range of a Banach-space operator
is much less rigid than in the Hilbert setting, and its topology is not
understood even on Banach spaces with very simple structure.
However, the following counterparts of Hilbertian properties remain true.

\begin{theorem}\label{ThW}
Let $T\in \mathcal L(X)$. Then:
\begin{enumerate}[label=\textup{(\roman*)}]
\item $\conv\sigma(T)\subset \overline{W(T)}$;
\item $\exp(-1)\|T\|\le w(T)\le \|T\|$;
\item $\overline{\conv}\, W(T)=V(T,\mathcal L(X))$.
\end{enumerate}
\end{theorem}
The properties stated in Theorem \ref{ThW} can be found in \cite{BD1, BD2}.
In particular, the result in (i) is due to Crabb, and is presented in \cite[Theorem 3, p.22]{BD2}.

In general, the numerical range is not convex in the Banach space setting, and thus to relate it to algebraic numerical range one has to pass to its convex hull. Moreover, the theory of Calkin algebras on Banach spaces is more involved 
and its interplay with the theory of (essential) numerical ranges depends on the geometric properties of the underlying Banach space. See Subsection \ref{genpropwe} for more on that.
The next statement records the properties of algebraic numerical range which can be formulated for general Banach spaces,
and can be found  in \cite[Chapter 3, \S9]{BD1} and \cite[Chapter 7, \S34]{BD2}.
\begin{theorem}\label{vrange}
Let $T\in \mathcal L(X).$ Then the following hold.
\begin{enumerate}[label=\textup{(\roman*)}]
\item $V(T,\mathcal L(X))=\overline{\conv}\, W(T)$;
\item  $V\bigl(\pi(T),\mathcal L(X)/\mathcal K(X)\bigr)
=\bigcap\{V(T+K,\mathcal L(X)): K\in \mathcal K(X)\},$
where \(\pi:\mathcal L(X)\to\mathcal L(X)/\mathcal K(X)\) is the quotient map.
\end{enumerate}
\end{theorem}

\subsection{Essential numerical range}\label{subsec:banach-essential-numerical-range}

We now pass from the ordinary numerical range to its essential counterpart. 
The first goal is to record the general Banach-space framework in a form suited to the later $\ell_p$ arguments.

We now define the main object of our study.
\begin{definition}
Let $X$ be a complex Banach space and $T\in \mathcal L(X)$. The essential numerical range $\We(T)$ is the set of all $\lambda\in\C$ such that there exists a net $(x_\alpha,x^*_\alpha)\subset\Pi(X)$ with
\[
x_\alpha\wto 0
\qquad\text{and}\qquad
\ip{Tx_\alpha}{x^*_\alpha}\to\lambda.
\]
\end{definition}

\begin{remark}\label{remark17}
If \(X\) is finite-dimensional, then \(W_e(T)=\varnothing\). If \(X\) is infinite-dimensional, the non-emptiness of \(W_e(T)\) is immediate from the net definition. Let $\mathcal U$ be the directed set of weak neighbourhoods of $0$, ordered by reverse inclusion. For each $U\in\mathcal U$, choose $(x_U,x_U^*)\in\Pi(X)$ with $x_U\in U$. Then $x_U\wto 0$, and the scalar net $\ip{Tx_U}{x_U^*}$ is bounded. Hence it has a convergent subnet, whose limit belongs to $\We(T)$.
Moreover,  by \cite[Proposition~5]{BM}, the set $\We(T)$ is closed, and is thus compact for arbitrary Banach-space operators.
\end{remark}

Note that \cite[Theorems~11 and~13]{BM} shows that in reflexive spaces, and more generally in Asplund spaces, the net definition may be replaced by an equivalent sequential one. 
On $\ell_p,$ $1<p<\infty$, one may therefore deal with sequences when convenient. In that setting closedness also follows by a standard diagonal argument.

A basic perturbation tool for the study of numerical ranges on Banach spaces is the following Bishop-Phelps-Bollob\'as (BPB) theorem, see e.g. \cite[Theorem 16.1]{BD2}.

\begin{theorem}\label{thm:bpb}
Given $\e>0$, let $x\in X$ and  $x^*\in X^*$ satisfy $\|x\|\le 1$, $\|x^*\|=1$, and
\[
|\ip{x}{x^*}-1|<\frac{\e^2}{4}.
\]
Then there exists $(y,y^*)\in\Pi(X)$ such that
\[
\|y-x\|<\e,
\qquad
\|y^*-x^*\|<\e.
\]
\end{theorem}

We shall use the following Banach-space characterisation of the essential numerical range.  The equivalence between the net definition and the finite-codimensional formulation is \cite[Proposition~6]{BM}. The intermediate approximate-normalisation form is included for later reference.

\begin{theorem}\label{thm:characterization-We}
 Let $T\in \mathcal L(X)$, and let $\lambda\in\C$. Then the following are equivalent.
\begin{enumerate}[label=\textup{(\roman*)}]
\item $\lambda\in\We(T)$.
\item There exists a net $(x_\alpha,x^*_\alpha)\in X\times X^*$ such that
\[
x_\alpha\wto 0,
\quad
\|x_\alpha\|\to 1,
\quad
\|x_\alpha^*\|\to 1,
\quad
\ip{x_\alpha}{x_\alpha^*}\to 1,
\quad
\ip{Tx_\alpha}{x_\alpha^*}\to \lambda.
\]
\item For every subspace $M\subset X$ with $\codim M<\infty$, and for every $\e>0$, there exists $(x,x^*)\in\Pi(X)$ with $x\in M$ and
\[
|\ip{Tx}{x^*}-\lambda|<\e.
\]
\end{enumerate}
\end{theorem}

\begin{proof}
The equivalence of \textup{(i)} and \textup{(iii)} is the finite-codimensional characterisation of the essential numerical range from \cite[Proposition~6]{BM}.  The implication \textup{(i)$\Rightarrow$(ii)} is immediate from the definition of $\We(T)$.

It remains only to justify that \textup{(ii)} implies \textup{(iii)}.  Let $M\subset X$ be finite-codimensional and let $\varepsilon>0$.  Put
\[
        L:=M\cap T^{-1}(M) .
\]
Then $L$ is finite-codimensional.  Choose a closed finite-dimensional complement $G$ of $L$ and let $Q:X\to G$ be the corresponding projection with $\Ker Q=L$.  Since $x_\alpha\wto0$ and $Q$ has finite rank, $Qx_\alpha\to0$ in norm.  Hence, for
\[
        u_\alpha:=x_\alpha-Qx_\alpha\in L,
\]
we have
\[
        \|u_\alpha-x_\alpha\|\to0,
        \qquad
        \|u_\alpha\|\to1 .
\]
Restrict $x_\alpha^*$ to $M$.  Since $u_\alpha\in L\subset M$ and $Tu_\alpha\in M$, the assumptions in \textup{(ii)} imply
\[
        \|x_\alpha^*|_M\|\to1,
        \qquad
        \ip{u_\alpha}{x_\alpha^*|_M}\to1,
        \qquad
        \ip{Tu_\alpha}{x_\alpha^*|_M}\to\lambda .
\]
For large $\alpha$, set
\[
        v_\alpha:=\frac{u_\alpha}{\|u_\alpha\|},
        \qquad
        f_\alpha:=\frac{x_\alpha^*|_M}{\|x_\alpha^*|_M\|}\in M^* .
\]
Then $\|v_\alpha\|=\|f_\alpha\|=1$ and
\[
        \ip{v_\alpha}{f_\alpha}\to1 .
\]
Next, apply the Bishop--Phelps--Bollob\'as theorem in the Banach space $M$ with an error tending to zero.  We obtain pairs $(y_\alpha,g_\alpha)\in\Pi(M)$ such that
\[
        \|y_\alpha-v_\alpha\|\to0,
        \qquad
        \|g_\alpha-f_\alpha\|\to0 .
\]
Let $y_\alpha^*\in X^*$ be a Hahn--Banach extension of $g_\alpha$.  Then $(y_\alpha,y_\alpha^*)\in\Pi(X)$ and $y_\alpha\in M$.  Since $u_\alpha\in L$, we have $Tu_\alpha\in M$. Hence
\[
\begin{aligned}
        \ip{Ty_\alpha}{y_\alpha^*}
        &=y_\alpha^*(T(y_\alpha-u_\alpha))+g_\alpha(Tu_\alpha).
\end{aligned}
\]
The first term tends to $0$.  Moreover,
\[
        |g_\alpha(Tu_\alpha)-f_\alpha(Tu_\alpha)|
        \le \|g_\alpha-f_\alpha\|\,\|Tu_\alpha\|\to0,
\]
and, since $\|x_\alpha^*|_M\|\to1$,
\[
        f_\alpha(Tu_\alpha)
        =
        \frac{\ip{Tu_\alpha}{x_\alpha^*|_M}}{\|x_\alpha^*|_M\|}
        \to\lambda .
\]
Thus $\ip{Ty_\alpha}{y_\alpha^*}\to\lambda$, and choosing $\alpha$ sufficiently large gives the pair required in \textup{(iii)}.
\end{proof}

\subsection{A Banach-space convex-hull inclusion for $\We(T)$}

We next relate the essential numerical range to spectral data. The resulting convex-hull inclusion is one of the Banach-space inputs that later reappears in the atomic $\ell_p$ setting.

For a Banach space $X$, we define the essential spectrum as in the Hilbert-space case, using the Calkin algebra. Thus, if $T\in\mathcal L(X)$ and $\pi:\mathcal L(X)\to\mathcal L(X)/\mathcal K(X)$ is the quotient map, then
\[
        \sigma_e(T):=\sigma\bigl(\pi(T),\mathcal L(X)/\mathcal K(X)\bigr).
\]

The essential approximate point spectrum $\sigma_{\pi e}(T)$ often arises in the study of numerical ranges in Hilbert spaces, 
and we will also need this notion for developing the Banach space theory. It can be defined in several equivalent ways, and  for the purposes of this paper
we set
\[
\sigma_{\pi e}(T):=
\left\{\lambda\in\C:
\begin{array}{l}
\forall M\subset X\text{ finite-codimensional},\ \forall\e>0,\\[1mm]
\exists x\in M\text{ with }\|x\|=1\text{ and }\|(T-\lambda I)x\|<\e.
\end{array}
\right\}
\]
Alternatively one may use approximate eigenvectors converging weakly to zero, resembling equivalent descriptions of $W_e(T)$ in 
Theorem \ref{thm:characterization-We}.

 We shall use the standard inclusions
\[
        \partial\sigma_e(T)\subset\sigma_{\pi e}(T)\subset\sigma_e(T),
\]
see e.g. \cite[Chapter III.19, Proposition 1]{MullerBook}.

Consequently, using the elementary fact that
\(\operatorname{conv}K=\operatorname{conv}\partial K\) for compact
sets \(K\subset\mathbb C\), we have
\begin{equation}\label{eq:conv-sigmae-sigmapie}
        \operatorname{conv}\sigma_e(T)
        =
        \operatorname{conv}\sigma_{\pi e}(T).
\end{equation}

The proof of the following theorem uses two standard auxiliary facts from
Banach space geometry. The first is a finite-codimensional separation device,
and the second is a classical convex-combination selection lemma due to
Zenger. Both are standard, see, for example, \cite[p.~329]{MullerBook}
and \cite[p.~20]{BD2}, respectively.

\begin{lemma}\label{lem:separation-finite-codim}
Let $N\subset X$ be a finite-dimensional subspace and let $0<\delta<1$. Then there exists a finite-codimensional subspace $M\subset X$ such that
\[
\|u+v\|\ge (1-\delta)\|u\|
\qquad (u\in N,\ v\in M).
\]
\end{lemma}

\begin{lemma}\label{lem:zenger}
Let $x_1,\dots,x_m$ be linearly independent vectors in $X,$
and let $a_1,\dots,a_m>0$ satisfy $\sum_{k=1}^m a_k=1$. Then there exist complex numbers $t_1,\dots,t_m$, a vector
\[
x=\sum_{k=1}^m t_kx_k\in X,
\]
and a functional $f\in X^*$ such that
\[
\|x\|=\|f\|=\ip{x}{f}=1,
\qquad
\ip{t_kx_k}{f}=a_k
\quad (1\le k\le m).
\]
\end{lemma}

We now combine these two lemmas with the definition of $\sigma_{\pi e}(T)$ to obtain an inclusion result for the convex hull of the essential spectrum,
which seems to have been missing from the literature. 
It is the essential-spectrum counterpart of Crabb's Theorem \ref{ThW},(i).

\begin{theorem}\label{thm:convhull-sigma-pie}
Let $X$ be a complex Banach space and let $T\in \mathcal L(X)$. Then
\[
\conv\sigma_{e}(T)\subset \We(T).
\]
\end{theorem}

\begin{proof}
By \eqref{eq:conv-sigmae-sigmapie}, it is enough to show that
\[
        \operatorname{conv}\sigma_{\pi e}(T)\subset W_e(T).
\]
Let
\[
        \lambda=\sum_{k=1}^m a_k\lambda_k,
        \qquad
        a_k>0,\quad \sum_{k=1}^m a_k=1,\quad
        \lambda_k\in\sigma_{\pi e}(T),
\]
and let \(M\subset X\) be a finite-codimensional subspace. We prove that,
for every \(\varepsilon>0\), there exists \((x,x^*)\in\Pi(X)\) such that
\(x\in M\) and
\[
        |\langle Tx,x^*\rangle-\lambda|<\varepsilon .
\]
The conclusion will then follow from
Theorem~\ref{thm:characterization-We}.

Choose \(0<\delta<1/2\), and put
\[
        C:=(1-\delta)^{-1},
        \qquad
        \eta:=\frac{\varepsilon}{2mC}.
\]
We construct unit vectors \(x_1,\ldots,x_m\in M\). Set \(L_0:=M\). Since
\(\lambda_1\in\sigma_{\pi e}(T)\), choose \(x_1\in L_0\) such that
\[
        \|x_1\|=1,
        \qquad
        \|(T-\lambda_1I)x_1\|<\eta .
\]
Suppose that \(x_1,\ldots,x_r\) have been chosen, where \(1\le r<m\), and
let
\[
        F_r:=\operatorname{span}\{x_1,\ldots,x_r\}.
\]
By Lemma~\ref{lem:separation-finite-codim}, there is a finite-codimensional
subspace \(N_r\subset X\) such that
\[
        \|u+v\|\ge (1-\delta)\|u\|
        \qquad (u\in F_r,\ v\in N_r).
\]
In particular, \(F_r\cap N_r=\{0\}\). Define
\[
        L_r:=M\cap N_1\cap\cdots\cap N_r .
\]
Then \(L_r\) is finite-codimensional. Since
\(\lambda_{r+1}\in\sigma_{\pi e}(T)\), choose \(x_{r+1}\in L_r\) such that
\[
        \|x_{r+1}\|=1,
        \qquad
        \|(T-\lambda_{r+1}I)x_{r+1}\|<\eta .
\]
This completes the recursive construction.

Since \(x_{r+1}\in N_r\) and \(F_r\cap N_r=\{0\}\) for \(1\le r<m\), the
vectors \(x_1,\ldots,x_m\) are linearly independent. Let
\[
        Y:=\operatorname{span}\{x_1,\ldots,x_m\}.
\]
For \(0\le r\le m\), let \(P_r:Y\to Y\) be the projection onto
\(\operatorname{span}\{x_1,\ldots,x_r\}\) along
\(\operatorname{span}\{x_{r+1},\ldots,x_m\}\), with \(P_0=0\) and \(P_m=I_Y\).
We claim that
\[
        \|P_r\|\le C \qquad (0\le r\le m).
\]
For \(r=0\) and \(r=m\) this is clear. Let \(1\le r<m\), and write
\[
        y=\sum_{k=1}^m \beta_kx_k\in Y
\]
with complex scalars \(\beta_k\). Then
\[
        P_ry\in F_r,
        \qquad
        y-P_ry=\sum_{k=r+1}^m \beta_kx_k\in N_r,
\]
because \(x_k\in L_{k-1}\subset N_r\) for \(k>r\). Hence
\[
        \|y\|\ge (1-\delta)\|P_ry\|,
\]
and the claim follows.

Let \(\varphi_1,\ldots,\varphi_m\in Y^*\) be the coordinate functionals
associated with the basis \(x_1,\ldots,x_m\). Since
\[
        (P_k-P_{k-1})y=\varphi_k(y)x_k,
        \qquad \|x_k\|=1,
\]
we have
\[
        |\varphi_k(y)|
        \le
        \bigl(\|P_k\|+\|P_{k-1}\|\bigr)\|y\|
        \le 2C\|y\|
        \qquad (1\le k\le m),
\]
and therefore
\[
        \|\varphi_k\|\le 2C
        \qquad (1\le k\le m).
\]

By Lemma~\ref{lem:zenger}, there exist complex scalars \(t_1,\ldots,t_m\), a
vector
\[
        x:=\sum_{k=1}^m t_kx_k\in Y,
\]
and a functional \(g\in Y^*\) such that
\[
        \|x\|=\|g\|=\langle x,g\rangle=1,
        \qquad
        \langle t_kx_k,g\rangle=a_k
        \quad (1\le k\le m).
\]
Let \(x^*\in X^*\) be a Hahn--Banach extension of \(g\) with
\(\|x^*\|=1\). Then
\[
        (x,x^*)\in\Pi(X),
        \qquad
        x\in Y\subset M.
\]
Moreover, since \(t_k=\varphi_k(x)\),
\[
        |t_k|\le 2C
        \qquad (1\le k\le m).
\]
Using
\[
        \langle t_kx_k,x^*\rangle
        =
        \langle t_kx_k,g\rangle
        =
        a_k
        \qquad (1\le k\le m),
\]
we obtain
\[
\begin{aligned}
        \langle Tx,x^*\rangle-\lambda
        &=
        \sum_{k=1}^m
        \langle t_k(T-\lambda_kI)x_k,x^*\rangle .
\end{aligned}
\]
Consequently,
\[
\begin{aligned}
        |\langle Tx,x^*\rangle-\lambda|
        &\le
        \sum_{k=1}^m
        |t_k|\,\|(T-\lambda_kI)x_k\|\,\|x^*\|  \\
        &\le
        \sum_{k=1}^m 2C\eta
        =
        \varepsilon .
\end{aligned}
\]
Since \(M\) and \(\varepsilon\) were arbitrary, 
 Theorem~\ref{thm:characterization-We} gives
\(\lambda\in W_e(T)\).
\end{proof}

As an immediate corollary, we are able to identify the essential numerical range of isometries with large peripheral spectrum,
in particular, of non-invertible isometries.
\begin{corollary}\label{cor:isometry-disk}
Let $T\in \mathcal L(X)$ be an isometry such that
\[
\mathbb T\subset \sigma(T).
\]
Then
\begin{equation}\label{isomwe}
\We(T)=\overline\D.
\end{equation}
In particular, the equality \eqref{isomwe} holds if $T$ is a non-invertible isometry.
\end{corollary}

\begin{proof}
For every $(x,x^*)\in \Pi(X)$ one has
\[
|\ip{Tx}{x^*}|\le \|Tx\|\,\|x^*\|=1,
\]
so
\[
\We(T)\subset \overline{W(T)}\subset \overline\D.
\]

Let $\lambda\in\mathbb T$. Since $\sigma(T)\subset \overline\D$ for every isometry and $\mathbb T\subset \sigma(T)$ by assumption, the point $\lambda$ belongs to $\partial\sigma(T)$ and is not isolated in $\sigma(T)$. Hence $\lambda\in \sigma_e(T)$, see, for example, \cite[Theorem~III.18.4]{MullerBook} and \cite[Proposition~3.7.8]{LaursenNeumann}. Therefore
\[
\mathbb T\subset \sigma_e(T).
\]
By Theorem~\ref{thm:convhull-sigma-pie},
\[
\overline\D=\conv\mathbb T\subset \conv\sigma_e(T)\subset \We(T).
\]
Thus $\We(T)=\overline\D$.

Finally, if $T$ is a non-invertible isometry, then $\sigma(T)=\overline{\mathbb D}$ by, for example, \cite[Corollary~6.14]{vanNeervenFA}. Hence the last claim follows.
\end{proof}

This concludes the Banach-space preparation used in the sequel. We now turn to the explicit $2\times2$ geometry on finite-dimensional $\ell_p$-spaces, which forms the first finite-dimensional layer of the atomic theory.

\section{Numerical ranges on finite-dimensional $\ell_p$-spaces}\label{secfinite}

This section gives the first finite-dimensional layer of the atomic theory. We work throughout with a fixed exponent $1<p<\infty$, and we write $q$ for the conjugate exponent, so that $1/p+1/q=1$. We keep the scalar spatial numerical range on finite-dimensional $\ell_p$ and isolate the concrete $2\times 2$ geometry that later underlies the atomic theory.
For $n\in\N$ we write $\ell_p^n$ for the $n$-dimensional $\ell_p$-space.
Let $e_1,\dots,e_n$ be the standard basis of $\ell_p^n$ and $f_1,\dots,f_n$ the standard basis of $\ell_q^n$.
For
\[
x=\sum_{j=1}^n \alpha_j e_j\in \ell_p^n,
\qquad
x^*=\sum_{j=1}^n \beta_j f_j\in \ell_q^n,
\]
we have
\begin{equation}\label{dual}
\ip{x}{x^*}=\sum_{j=1}^n \alpha_j\overline{\beta_j}.
\end{equation}
Throughout this section, matrices are identified with the corresponding
operators with respect to the standard bases.

If $x\in \ell_p^n$ is a unit vector, then
\[
J(x)=\sum_{j=1}^n \alpha_j|\alpha_j|^{p-2}f_j
\qquad\text{for }x=\sum_{j=1}^n \alpha_j e_j,
\]
is the unique norming functional of $x$. Hence, for $T\in \mathcal L(\ell_p^n)$,
\[
W(T)=\{\ip{Tx}{J(x)}: x\in \ell_p^n,\ \|x\|=1\}.
\]

\subsection{The $2\times 2$ parametrisation}

This parametrisation is the basic computational tool for the finite-dimensional
part of the paper.

\begin{proposition}\label{prop:lp2-reduction}
Let $T\in\mathcal L(\ell_p^2)$ be represented as
\[
T=
\begin{pmatrix}
a & b\\
c & d
\end{pmatrix}
\]
with $a,b,c,d \in \mathbb C.$ Set
\[
\tau:=\frac{a+d}{2},\qquad
\alpha:=\frac{a-d}{2},
\]
\[
\beta(r):=r^{1/q}(1-r)^{1/p},\qquad
\gamma(r):=r^{1/p}(1-r)^{1/q}.
\]
For $0\le r\le1$ put
\begin{equation}\label{er}
E_r:=\Bigl\{(2r-1)\alpha+\beta(r)be^{i\theta}
        +\gamma(r)ce^{-i\theta}:\theta\in [0,2\pi]\Bigr\}.
\end{equation}
Then
\[
W(T)=\tau+\bigcup_{0\le r\le1}E_r.
\]
\end{proposition}

\begin{proof}
Let $x=(u,v)\in\ell_p^2$ be a unit vector. Multiplying $x$ by a scalar of modulus one does not change the value of
$\ip{Tx}{J(x)}$, because $J(\omega x)=\omega J(x)$ for $|\omega|=1$. Hence, after such a rotation, we may write
\[
u=r^{1/p},\qquad v=e^{i\theta}(1-r)^{1/p}
\]
with $0\le r\le1$ and $\theta\in [0,2\pi].$ For this vector,
\[
J(x)=\bigl(r^{1/q},e^{i\theta}(1-r)^{1/q}\bigr).
\]
Using \eqref{dual}, 
we obtain
\begin{align*}
\ip{Tx}{J(x)}
&=\bigl(ar^{1/p}+be^{i\theta}(1-r)^{1/p}\bigr)r^{1/q}  \\
&\quad+\bigl(cr^{1/p}+de^{i\theta}(1-r)^{1/p}\bigr)e^{-i\theta}(1-r)^{1/q} \\
&=ar+d(1-r)+b\,r^{1/q}(1-r)^{1/p}e^{i\theta}
   +c\,r^{1/p}(1-r)^{1/q}e^{-i\theta} \\
&=\tau+(2r-1)\alpha+\beta(r)be^{i\theta}+
   \gamma(r)ce^{-i\theta}.
\end{align*}
This proves the formula for $W(T)$.
\end{proof}

\begin{remark}\label{rem:lp2-ellipses}
For fixed $r$, the set $E_r$ is a possibly degenerate ellipse with centre
$(2r-1)\alpha$. More precisely, write
\[
b=Be^{i\phi},\qquad c=Ce^{i\psi},\qquad
B:=|b|,
\qquad C:=|c|,
\qquad \sigma:=\frac{\phi+\psi}{2}.
\]
Then
\[
\begin{aligned}
E_r=(2r-1)\alpha+e^{i\sigma}
\Bigl\{&\bigl(B\beta(r)+C\gamma(r)\bigr)\cos\eta\\
&+i\bigl(B\beta(r)-C\gamma(r)\bigr)\sin\eta:
\eta\in [0,2\pi]
\Bigr\}.
\end{aligned}
\]
Indeed, this follows from the change of variable
$\eta=\theta+(\phi-\psi)/2$. After the rotation by $e^{-i\sigma}$ and
translation by $-(2r-1)\alpha$, the curve has the form
$A_r\cos\eta+iD_r\sin\eta$, with $A_r,D_r\in\R$. Hence it is an ellipse,
with the usual degeneracies allowed.
\end{remark}

\begin{remark}\label{rem:p2-elliptic-range}
In the Hilbertian case $p=2$, the preceding parametrisation reduces to the usual elliptical-range picture. Indeed, then
\[
\beta(r)=\gamma(r)=\sqrt{r(1-r)}.
\]
Writing $t=2r-1$, we have
\[
\sqrt{r(1-r)}=\frac12\sqrt{1-t^2}.
\]
Using the notation of Remark~\ref{rem:lp2-ellipses}, after subtracting $\tau$ we obtain
\[
W(T)-\tau=
\left\{
 t\alpha+\frac12 e^{i\sigma}
 \bigl((B+C)u+i(B-C)v\bigr):
 t^2+u^2+v^2=1
\right\}.
\]

Thus, letting $S^2$ be the unit sphere in $\mathbb R^3,$ we have \(W(T)-\tau=L(S^2)\) for a real-linear map \(L:\mathbb R^3\to\mathbb C\).
Since \(L\) has non-trivial kernel, \(L(S^2)\) is the same as the image under
\(L\) of the closed Euclidean unit ball of \(\mathbb R^3\). Indeed, if
\(y=Lx\) with \(\|x\|_2<1\), one may add a suitable multiple of a non-zero
vector in \(\ker L\) to reach the unit sphere without changing \(Lx\).
Therefore this image is an elliptical disc, possibly degenerate.
This recovers the classical elliptic range theorem for $2\times2$ Hilbert-space operators.
\end{remark}

Using Proposition~\ref{prop:lp2-reduction} it is easy to produce a plethora of operators on $\ell_p^2(\C)$ with non-convex numerical range, and we give a sample below. Another example
of such an operator was given in \cite[Section 11]{BD1}, and similar examples were provided later in \cite[Example 11.5]{AppDV} and \cite{Mandal}.

\begin{example}\label{ex:lp2-nonconvex}
Let $p\neq 2$. Then there exists an operator $T\in \mathcal L(\ell_p^2)$ such that $W(T)$ is not convex.
\end{example}

\begin{proof}
Let $T$ be defined as 
\[
T:=\begin{pmatrix}1&1\\1&-1\end{pmatrix}.
\]
For this matrix, in the notation of Proposition~\ref{prop:lp2-reduction}, we have
\[
        \tau=0,\qquad \alpha=1,\qquad b=c=1.
\]
Thus Proposition~\ref{prop:lp2-reduction} gives
\[
        W(T)=\bigcup_{0\le r\le1}E_r,
\]
where
\[
        E_r=
        \left\{
        (2r-1)+(\beta(r)+\gamma(r))\cos\theta
        +i(\beta(r)-\gamma(r))\sin\theta:
        \theta\in [0,2\pi]
        \right\}.
\]
Let
\[
M:=\max\{\beta(r)-\gamma(r):0\le r\le 1\}.
\]

Since
\[
\beta(r)-\gamma(r)
=
r^{1/p}(1-r)^{1/p}
\bigl(r^{1-2/p}-(1-r)^{1-2/p}\bigr),
\]
and \(p\ne 2\), we have \(M>0\). Moreover,
\[
        \beta(1-r)-\gamma(1-r)=-(\beta(r)-\gamma(r)),
\]
and therefore
\[
        M=\max_{0\le r\le 1}|\beta(r)-\gamma(r)|.
\]
Choose \(r_0\in[0,1]\) such that
\[
        M=\beta(r_0)-\gamma(r_0).
\]
Then, for $\theta=\pi/2$,
\[
(2r_0-1)+iM\in E_{r_0}\subset W(T).
\]
Also,
\[
\beta(1-r_0)-\gamma(1-r_0)=-M,
\]
so for $\theta=3\pi/2$ we get
\[
1-2r_0+iM\in E_{1-r_0}\subset W(T).
\]
We claim that $iM\notin W(T)$. More generally, if $is\in W(T)$ for some $s\in\R$, then $|s|<M$.
Suppose on the contrary that there exist $r_1\in[0,1]$ and $s\in\R$ with $|s|\ge M$ and $is\in E_{r_1}$. Then
\[
s=(\beta(r_1)-\gamma(r_1))\sin\theta
\]
for some $\theta\in[0,2\pi]$. Hence $|s|=M$ and $\sin\theta=\pm1$, so $\cos\theta=0$. Therefore
\[
is=2r_1-1+i(\beta(r_1)-\gamma(r_1))\sin\theta.
\]
It follows that $r_1=1/2$, and hence $\beta(r_1)-\gamma(r_1)=0$, so $s=0$, a contradiction.
Thus $W(T)$ is not convex.
\end{proof}

\subsection{Convex hull of diagonal entries}
\leavevmode\par

Having identified the geometry of numerical ranges of operators on $\ell_p^2$ explicitly, we turn to the first positive structural feature. The next results show that diagonal data already enforce nontrivial convexity inside the numerical range.

We start with a lemma supplying the ingredients needed for the study of diagonals in the $\ell_p$-setting.
It is based on an elementary homotopy argument, and
may be of independent interest. 

For a compact set $K\subset\C$, we write $\widehat K$ for its polynomial hull. In the plane this is the union of $K$ with all bounded components of $\C\setminus K$. The concept enters here because the ellipses $E_s$ are boundary ellipses, whereas the diagonal segment argument requires access to the filled region enclosed by such an ellipse.

With \(E_r\) as in \eqref{er}, the following lemma says that the filled
ellipse bounded by \(E_s\) is swept out by the later ellipses \(E_r\),
\(s\le r\le1\).
\begin{lemma}\label{lem:lp2-hull-step}
 Let \(T\in \mathcal L(\ell_p^2)\) be given by
\[
        T=
        \begin{pmatrix}
        a & b\\
        c & d
        \end{pmatrix},
\]
and let \(s\in(0,1)\). Set \(\tau:=(a+d)/2\). If
\[
        \lambda\in \tau+\widehat{E_s},
\]
then
\[
        \lambda\in \tau+\bigcup_{s\le r\le1}E_r\subset W(T).
\]
\end{lemma}
\begin{proof}
We use the description of $W(T)$ given by Proposition \ref{prop:lp2-reduction},
and assume without loss of generality that $\tau=0$.


The statement is clear if \(\lambda\in\bigcup_{s\le r\le1}E_r\). Assume that $\lambda\in \Int \widehat{E_s}$ and suppose, towards a contradiction, that
\[
\lambda\notin \bigcup_{s\le r\le1}E_r. 
\]
For $r\in[s,1]$ define a curve $f_r:[0,2\pi]\to\C$ by
\[
f_r(\theta)=(2r-1)\alpha+\beta(r)be^{i\theta}+\gamma(r)ce^{-i\theta}.
\]
Then $f_r([0,2\pi])=E_r$, so $f_r(\theta)\ne \lambda$ for all $r\in[s,1]$ and $\theta\in[0,2\pi]$.
Moreover, $f_1$ is a constant curve, whereas
\[
\operatorname{ind}(f_s,\lambda)\ne 0
\qquad\text{and}\qquad
\operatorname{ind}(f_1,\lambda)=0.
\]
But $(r,\theta)\mapsto f_r(\theta)$ is a homotopy in $\C\setminus\{\lambda\}$ between $f_s$ and $f_1$, a contradiction.
Therefore
\(
\lambda\in \bigcup_{s\le r\le1}E_r,
\)
and Proposition \ref{prop:lp2-reduction} implies that $\lambda \in W(T).$
\end{proof}

\begin{corollary}\label{cor:lp2-diagonal-segment}
Let $T \in \mathcal L(\ell_p^2)$ be represented as
\[
T=
\begin{pmatrix}
a & b\\
c & d
\end{pmatrix},
\]
where $a, b, c, d \in \mathbb C.$ Then
\[
\conv\{a,d\}\subset W(T).
\]
\end{corollary}

\begin{proof}
The endpoints $a$ and $d$ are attained at the basis vectors. Let
$\tau=(a+d)/2$ and $\alpha=(a-d)/2$. Let
$0<\varepsilon<1$. Since
\[
        (1-\varepsilon)a+\varepsilon d
        =
        \tau+(1-2\varepsilon)\alpha
        \in \tau+\widehat E_{1-\varepsilon},
\]
Lemma~\ref{lem:lp2-hull-step} gives
$(1-\varepsilon)a+\varepsilon d\in W(T)$. Hence the whole segment
$\operatorname{conv}\{a,d\}$ is contained in $W(T)$.
\end{proof}

We shall also use the following elementary observation. Let \(X\) be a finite
or infinite coordinate \(\ell_p\)-space, and let
\(u_1,\ldots,u_m\in X\) be unit vectors with pairwise disjoint supports.
 Put
\[
        E:=\operatorname{span}\{u_1,\ldots,u_m\},
\]
and let \(V:\ell_p^m\to E\) be the isometry
\[
        V(\xi_1,\ldots,\xi_m)=\sum_{j=1}^m \xi_j u_j .
\]
For \(T\in\mathcal L(X)\), define \(T_E\in\mathcal L(\ell_p^m)\) by
\[
        T_Ee_j=
        \sum_{i=1}^m \langle Tu_j,J(u_i)\rangle e_i,
        \qquad j=1,\ldots,m.
\]
We call \(T_E\) the \emph{block compression} of \(T\) associated with
\(u_1,\ldots,u_m\).
Observe that
\begin{equation}\label{eq:block-compression}
        W(T_E)\subset W(T).
\end{equation}
Indeed, if \(\xi=(\xi_j)_{j=1}^m\) is a unit vector in \(\ell_p^m\) and
\(x=V\xi\), then \(\|x\|=1\) and, because the supports of the \(u_j\)'s are
pairwise disjoint,
\[
        J(x)=\sum_{j=1}^m \xi_j|\xi_j|^{p-2}J(u_j).
\]
Consequently,
\[
        \langle T_E\xi,J(\xi)\rangle=\langle Tx,J(x)\rangle\in W(T),
\]
which proves \eqref{eq:block-compression}.

We can now identify subsets of numerical ranges in terms of convex hulls of
diagonal entries, although the \(\ell_p\)-numerical range is not convex in general.
The result given below will also be useful for obtaining its  countable convex-combination counterpart proved in Section~\ref{inflp}.

\begin{theorem}\label{thm:finite-diagonal-convexity}
Let $n\in\N$ and let $T\in \mathcal L(\ell_p^n)$. Then
\[
\conv\bigl\{\ip{Te_j}{J(e_j)}: j=1,\dots,n\bigr\}\subset W(T).
\]
\end{theorem}

\begin{proof}
The statement is clear if $n=1$. For $n=2$ it is precisely Corollary~\ref{cor:lp2-diagonal-segment}.

Assume now that $n\ge 3$ and that the statement has been proved for $n-1$.
Let $\alpha_1,\dots,\alpha_n\ge 0$ satisfy $\sum_{j=1}^n \alpha_j=1$.
We must show that
\[
\sum_{j=1}^n \alpha_j\ip{Te_j}{J(e_j)}\in W(T).
\]
If $\alpha_n=1$, there is nothing to prove. Set
\[
\beta:=\sum_{j=1}^{n-1}\alpha_j=1-\alpha_n>0.
\]
By the induction hypothesis, there exists a unit vector
\[
x\in \operatorname{span}\{e_1,\dots,e_{n-1}\}
\]
such that
\[
\ip{Tx}{J(x)}=\sum_{j=1}^{n-1}\frac{\alpha_j}{\beta}\,\ip{Te_j}{J(e_j)}.
\]
Let
\[
X:=\operatorname{span}\{x,e_n\}.
\]
Since $x$ and $e_n$ have disjoint supports, $X$ is isometrically isomorphic to $\ell_p^2$. Let $R$ be the block compression of $T$ to $X$ with respect to the ordered basis $(x,e_n)$ in the sense of the preceding observation. Its diagonal entries are
\[
\ip{Tx}{J(x)}
\qquad\text{and}\qquad
\ip{Te_n}{J(e_n)}.
\]
Applying Corollary~\ref{cor:lp2-diagonal-segment} to $R$, and then using the inclusion \eqref{eq:block-compression},
we obtain a unit vector $y\in X$ such that
\[
\ip{Ty}{J(y)}=\beta\,\ip{Tx}{J(x)}+\alpha_n\,\ip{Te_n}{J(e_n)}.
\]
Hence
\[
\ip{Ty}{J(y)}=\sum_{j=1}^{n-1}\alpha_j\ip{Te_j}{J(e_j)}+\alpha_n\ip{Te_n}{J(e_n)}=\sum_{j=1}^{n}\alpha_j\ip{Te_j}{J(e_j)},
\]
as required.
\end{proof}

\subsection{Positive and negative $2\times2$ geometry}
\leavevmode\par

The remaining two-dimensional statements illustrate both sides of the finite-dimensional picture: there are positive star-shapedness results for special matrices, but
there is also a counterexample to a general trace-centre theorem
for $2\times2$ matrices.

Although finite-dimensional numerical ranges need not be convex, numerical ranges of some natural families in $\mathcal L(\ell_p^2)$ remain star-shaped, while the trace centre need not be a star-centre, already on $\ell_3^2$.

\begin{proposition}\label{prop:lp2-triangular-star}
Let $T \in \mathcal L(\ell_p^2(\C))$ have matrix representation 
\[
T=\begin{pmatrix} a & b \\ 0 & d \end{pmatrix}
\]
with $\{a, b, d\} \subset \mathbb C.$
Then $W(T)$ is star-shaped with centre
\[
        \frac{\operatorname{tr}T}{2}=\frac{a+d}{2}.
\]
\end{proposition}

\begin{proof}
 Set
\[
\tau:=\frac{a+d}{2},
\qquad
\alpha:=\frac{a-d}{2},
\qquad
S:=T-\tau I=
\begin{pmatrix}
\alpha & b\\
0 & -\alpha
\end{pmatrix}.
\]
Then $W(T)=\tau+W(S)$, so it suffices to prove that $W(S)$ is star-shaped with centre $0$.

If $\alpha=0$, then by Proposition~\ref{prop:lp2-reduction},
\[
W(S)=\{\rho(r)e^{i\theta}:0\le r\le 1,\ \theta\in [0,2\pi]\},
\qquad
\rho(r):=|b|\,r^{1/q}(1-r)^{1/p}.
\]
Since $\rho$ is continuous on $[0,1]$ with $\rho(0)=\rho(1)=0$, its range is an interval $[0,\rho_{\max}]$. Hence
\[
W(S)=\{\zeta\in\C:|\zeta|\le \rho_{\max}\},
\]
a closed disc centred at $0$.

Assume now that $\alpha\ne 0$. Multiplying $S$ by a unimodular scalar only rotates $W(S)$, so without loss of generality we may assume that $\alpha>0$ is real. By Proposition~\ref{prop:lp2-reduction},
\[
W(S)=\{c(r)+\rho(r)e^{i\theta}:0\le r\le 1,\ \theta\in [0,2\pi]\},
\]
where
\[
c(r):=(2r-1)\alpha,
\qquad
\rho(r):=|b|\,r^{1/q}(1-r)^{1/p}.
\]
Thus, for each fixed \(r\), the corresponding curve is the circle with centre
\(c(r)\) and radius \(\rho(r)\).

The function $\rho$ is concave on $[0,1]$: the map $(u,v)\mapsto u^{1/q}v^{1/p}$ is concave on $\R_+^2$ because the exponents $1/q$ and $1/p$ are positive and sum to $1$, hence $r\mapsto r^{1/q}(1-r)^{1/p}$ is concave on $[0,1]$.

Let
\[
z=c(r)+\rho(r)e^{i\theta}\in W(S)
\]
and let $t\in[0,1]$. Define
\[
r_t:=tr+\frac{1-t}{2}.
\]
Then $c(r_t)=t\,c(r)$. By concavity of $\rho$,
\[
\rho(r_t)\ge t\rho(r)+(1-t)\rho(1/2)\ge t\rho(r).
\]
Consider the continuous function $F: [0,1]\to \mathbb R$ defined by
\[
F(s):=|tz-c(s)|-\rho(s),\qquad s\in[0,1].
\]
At $s=r_t$ we have
\[
F(r_t)=|tz-c(r_t)|-\rho(r_t)=t\rho(r)-\rho(r_t)\le 0,
\]
whereas
\[
F(0)=|tz+\alpha|\ge 0.
\]
Hence, by continuity, there exists $s_0\in[0,r_t]$ such that $F(s_0)=0$. Thus
\[
|tz-c(s_0)|=\rho(s_0),
\]
so for some $\varphi\in\R$,
\[
tz=c(s_0)+\rho(s_0)e^{i\varphi}\in W(S).
\]
Therefore $tW(S)\subset W(S)$, and $W(S)$ is star-shaped with centre $0$.
Consequently, $W(T)=\tau+W(S)$ is star-shaped with centre $\tau=(a+d)/2$.
\end{proof}

\begin{remark}
By symmetry, the same conclusion holds for lower triangular matrices
\[
T=\begin{pmatrix} a & 0 \\ c & d \end{pmatrix}.
\]
\end{remark}
We do not know whether $W(T)$ is star-shaped for every $T \in \mathcal L(\ell_p^2).$
We can only show that some natural choices of its points
may not be star-centres, and even that task can be demanding as the next proposition shows.
\begin{proposition}\label{prop:lp2-counterexample-full}
Consider $S \in \mathcal L (\ell_3^2)$ represented as
\[
S=
\begin{pmatrix}
8+\frac{i}{5} & 1\\[1mm]
\frac65 & -8-\frac{i}{5}
\end{pmatrix}.
\]
Then
\[
\operatorname{tr}S=0,
\qquad
0\in W(S),
\]
but $W(S)$ is \emph{not} star-shaped with centre $0$. More precisely,
\[
6\in W(S),
\qquad
3\notin W(S).
\]
Hence
\[
W(S)\cap [0,\infty)
\]
is not an interval, so $W(S)$ cannot be star-shaped with centre $0$.
\end{proposition}

\begin{proof}
The full computation is rather technical and is deferred to
Appendix~\ref{app:lp2-counterexample-proof}. There we analyse the positive
real ray in the $t$-parametrisation of $W(S)$ and prove that
$3\notin W(S)$, while $0,6\in W(S)$. Since $3$ is the midpoint of $0$ and
$6$, the conclusion follows.
\end{proof}

\begin{remark}\label{rem:lp2-padding-fails}
Let \(T_0\in\mathcal L(\ell_p^2)\), and let \(m\ge1\). Consider the block operator
\[
 T:=T_0\oplus 0_m
\]
acting on $\ell_p^{2+m}=\ell_p^2\oplus_p\ell_p^m$. Then
\[
W(T)=\bigcup_{s\in[0,1]} sW(T_0).
\]
In particular, if $0\in W(T_0)$, then $W(T)$ is automatically star-shaped with centre $0$, regardless of the geometry of $W(T_0)$. More generally, for a scalar $c\in\C$,
\[
W(T_0\oplus cI_m)=c+\bigcup_{s\in[0,1]} s\bigl(W(T_0)-c\bigr).
\]
Hence padding by zero blocks or scalar blocks radialises the $2\times2$ geometry and destroys the mechanism of Proposition~\ref{prop:lp2-counterexample-full}. Extending that proposition to arbitrary dimension therefore requires a genuinely nontrivial construction rather than a simple block embedding.
\end{remark}

The finite-dimensional picture is therefore mixed: explicit non-convexity and
failure of trace-centred star-shapedness coexist with convexity phenomena coming
from diagonal data. The next section shows that, on \(\ell_p\), the tail
version of this diagonal mechanism feeds into the essential numerical range
and the star-centre geometry of \(W(T)\).

\section{Numerical and essential numerical ranges on $\ell_p$}\label{inflp}

In this section we pass from finite-dimensional geometry of \(\ell_p^n\) to
the geometry of \(\ell_p\). In particular, the new feature is the
essential numerical range: it is described by tail subspaces, and this tail
description will be the basis for the convexity and Calkin-algebra results
below.

We continue to work with the same fixed exponent $1<p<\infty$, and we write $q$ for the conjugate exponent. We first record the basic continuity properties of the duality map and the persistence of non-convexity, then prove the countable diagonal convex-combination theorem, and finally develop the structure theory of the essential numerical range and its star-centre consequences.

Let $(e_j)_{j \ge 1}$ and $(f_j)_{j \ge 1}$ be the standard bases in $\ell_p$ and $\ell_q$, respectively.
For
\[
x=\sum_{j=1}^\infty \alpha_j e_j\in \ell_p,
\]
write
\[
J(x):=\sum_{j=1}^\infty \alpha_j|\alpha_j|^{p-2}f_j.
\]
Then
\[
\|J(x)\|_q^q=\|x\|_p^p.
\]
In particular, if $x\in\ell_p$ is a unit vector, then $J(x)$ is the unique functional satisfying
\[
\|J(x)\|_q=1=\ip{x}{J(x)}.
\]
In the concrete sequence-space formulas below, the coordinate adjoint is taken with respect to this sesquilinear pairing; thus it is represented by the conjugate-transpose matrix acting on $\ell_q$.
For $m\in\N$ let $P_m$ be the canonical projection onto $\bigvee_{j=1}^m e_j$ defined by $P_me_j=e_j\quad(j\le m)$ and $P_me_j=0\quad(j>m)$.

\subsection{Continuity of $J$ and persistence of non-convexity}

\begin{lemma}\label{thm:J-ellp}
Let $x,x_n\in \ell_p$ for $n\in\N$. Then:
\begin{enumerate}[label=\textup{(\roman*)}]
\item if $\|x_n-x\|_p\to 0$, then $\|J(x_n)-J(x)\|_q\to 0$;
\item if $x_n\wto x$ in $\ell_p$, then $J(x_n)\wto J(x)$ in $\ell_q$.
\end{enumerate}
\end{lemma}

\begin{proof}
Both assertions are immediate from the coordinate formula for $J$, see also
\cite[Theorems~2.16 and~4.14]{Cio}.
\end{proof}

In bigger spaces, one might hope for better convexity properties of the numerical range, but non-convexity still persists on $\ell_p$.

\begin{example}\label{ex:nonconvex-infinite-ellp}
Let $p\ne 2$ and let $T:\ell_p\to\ell_p$ be defined by
\[
Te_1=e_1+e_2,
\qquad
Te_2=e_1-e_2,
\qquad
Te_j=0 \quad (j\ge 3).
\]
Then $W(T)$ is not convex.
\end{example}

\begin{proof}
Let $S \in \mathcal L(\ell_p^2)$ be given by
\[
S=
\begin{pmatrix}
1&1\\
1&-1
\end{pmatrix}.
\]
 As in Example~\ref{ex:lp2-nonconvex}, one obtains two points $a+iM$ and $-a+iM$ in $W(S)$, for some $a\in\R$, whose midpoint $iM$ does not belong to $W(S)$.
For $x\in\ell_p$ with $\|x\|=1$, put
\[
        r:=\bigl(|x_1|^p+|x_2|^p\bigr)^{1/p}.
\]
If $r>0$, then $u:=r^{-1}(x_1,x_2)$ is a unit vector in $\ell_p^2$, and
\[
        \langle Tx,J(x)\rangle
        = r^p\langle Su,J(u)\rangle.
\]
The case $r=0$ gives the value $0$. Conversely, every value
$t\langle Su,J(u)\rangle$, where $0\le t\le1$ and $u$ is a unit vector in
$\ell_p^2$, is obtained by taking the first two coordinates equal to
$t^{1/p}u$ and adding a tail of norm $(1-t)^{1/p}$. Hence
\[
W(T)=\{tz:z\in W(S),\ 0\le t\le 1\}.
\]
 Hence $a+iM,-a+iM\in W(S)\subset W(T)$. It was shown in Example~\ref{ex:lp2-nonconvex} that
no point $is \in W(S)$ satisfies $|s| \ge M.$
So $iM\notin W(T)$ and $W(T)$ is not convex.
\end{proof}
This is compatible with Remark~\ref{rem:lp2-padding-fails}: padding radialises
the two-dimensional numerical range, but radialisation need not restore
convexity; the particular two-dimensional gap used here survives after
radialisation.

\subsection{Countable convex combinations of diagonal entries}

The next surprising theorem is a characteristic atomic feature: the numerical range contains all countable convex combinations of the diagonal entries.

We shall use the following elementary fact, in the planar form recorded in~\cite{RubinWesler1958}. If $C$ is a convex subset of $\mathbb C$, $z_j\in C$, $\alpha_j\ge0$, $\sum_j\alpha_j=1$, and the series $\sum_j\alpha_j z_j$ converges, then $\sum_j\alpha_j z_j\in C$. 

\begin{theorem}\label{thm:countable-diagonal-convex}
Let $T\in \mathcal L(\ell_p)$, let $\alpha_j\ge 0$ for $j\in\N$, and assume that
\[
\sum_{j=1}^\infty \alpha_j=1.
\]
Then
\[
\sum_{j=1}^\infty \alpha_j\ip{Te_j}{J(e_j)}\in W(T).
\]
\end{theorem}

\begin{proof}
Put
\[
        d_j:=\ip{Te_j}{J(e_j)},\qquad j\in\N.
\]
Since $|d_j|\le \|T\|$, the series $\sum_{j=1}^\infty\alpha_jd_j$ converges absolutely. Let
\[
        C:=\conv\{d_j:j\in\N\}\subset\C.
\]
By the preceding fact,
\[
        d:=\sum_{j=1}^\infty\alpha_jd_j\in C.
\]
Thus $d$ belongs to the convex hull of the set $\{d_j:j\in\N\}$. Hence there are indices $j_1,\ldots,j_m$ and numbers $\beta_1,\ldots,\beta_m\ge0$ such that
\[
        \sum_{k=1}^m\beta_k=1
\]
and
\[
        d=\sum_{k=1}^m\beta_k d_{j_k}.
\]
Let
\[
        F:=\{j_1,\ldots,j_m\},
        \qquad E_F:=\operatorname{span}\{e_j:j\in F\}.
\]
Let $T_F$ be the coordinate compression of $T$ to $E_F$, written as an operator on $\ell_p^{|F|}$ after the natural identification of $E_F$ with $\ell_p^{|F|}$. By Theorem~\ref{thm:finite-diagonal-convexity}, applied to $T_F$,
\[
        d\in W(T_F).
\]
The coordinate case of the block-compression observation gives $W(T_F)\subset W(T)$. Hence $d\in W(T)$, which proves the theorem.
\end{proof}

\subsection{The essential numerical range on $\ell_p$}

We now pass from the ordinary numerical range on $\ell_p$ to its essential counterpart. In the atomic setting the abstract Banach-space descriptions become considerably more concrete, because finite-codimensional conditions may be expressed in terms of tails and finite supports.

For later use we spell out how the preceding Banach-space definition reads on $\ell_p$. Since $\ell_q$ is separable, the net in the definition of $\We(T)$ may be replaced by a sequence. Moreover, the norming functional of a unit vector $x\in\ell_p$ is unique and equals $J(x)$. Hence
\[
\lambda\in\We(T)
\]
is equivalent, in the present setting, to the existence of a sequence of unit vectors $(x_k)\subset\ell_p$ such that
\[
        x_k\wto0,
        \qquad
        \ip{Tx_k}{J(x_k)}\to\lambda .
\]
No new definition is being introduced here.  The next theorem sharpens this sequential form by using coordinate tails and finite supports.
\begin{theorem}\label{thm:We-ellp-tail-char}
Let $T\in \mathcal L(\ell_p)$ and let $\lambda\in\C$. Then the following conditions are equivalent.
\begin{enumerate}[label=\textup{(\roman*)}]
\item $\lambda\in \We(T)$.
\item For every $k\in\N$ and every $\varepsilon>0$, there exists a unit vector $x$ with finite support such that
\[
\supp x\subset\{k+1,k+2,\dots\},
\qquad
|\ip{Tx}{J(x)}-\lambda|<\varepsilon.
\]
\item For all subspaces $M,F\subset \ell_p$ with $\codim M<\infty$ and $\dim F<\infty$, and for every $\varepsilon>0$, there exists a unit vector $x\in M$ with finite support such that
\[
|\ip{Tx}{J(x)}-\lambda|<\varepsilon,
\qquad
\sup\{|\ip{f}{J(x)}|: f\in F,\ \|f\|=1\}<\varepsilon.
\]
\end{enumerate}
\end{theorem}

\begin{proof}
\textup{(i)$\Rightarrow$(ii).} By
Theorem~\ref{thm:characterization-We}, there exists a unit vector $x\in (I-P_k)\ell_p$ such that $|\ip{Tx}{J(x)}-\lambda|<\varepsilon$.
For $m\ge k$ large enough, $P_mx\ne 0$. For such $m$, set
\[
x_m:=\frac{P_mx}{\|P_mx\|}.
\]
Then $x_m\to x$ in norm, and Lemma~\ref{thm:J-ellp}(i) gives $J(x_m)\to J(x)$ in norm. Hence for $m$ large enough,
\[
|\ip{Tx_m}{J(x_m)}-\lambda|<\varepsilon,
\]
and $x_m$ has finite support contained in $\{k+1,k+2,\dots\}$.

\smallskip

\textup{(ii)$\Rightarrow$(iii).} Let $M,F\subset \ell_p$ with $\codim M<\infty$ and $\dim F<\infty$, and let $\varepsilon>0$.
Choose a sequence $(x_k)$ as in \textup{(ii)}, with
\[
\supp x_k\subset\{k+1,k+2,\dots\},
\qquad
|\ip{Tx_k}{J(x_k)}-\lambda|<k^{-1}
\qquad (k\in\N).
\]
Then $x_k\wto 0$.

Let $Q:\ell_p\to \ell_p/M$ be the quotient map. Since $\ell_p/M$ is finite-dimensional and $Qe_1,Qe_2,\dots$ span $\ell_p/M$, there exist indices $j_1,\dots,j_m\in\N$ such that $Qe_{j_1},\dots,Qe_{j_m}$ form a basis of $\ell_p/M$. Set
\[
E:=\operatorname{span}\{e_{j_1},\dots,e_{j_m}\}.
\]
Then $Q|_E:E\to \ell_p/M$ is an isomorphism. Let
\[
R:=(Q|_E)^{-1}Q\in \mathcal L(\ell_p,E).
\]
Thus $R$ has finite rank and
\[
x-Rx\in M
\qquad (x\in\ell_p).
\]
For each $k$, put
\[
y_k:=x_k-Rx_k\in M.
\]
Since $R$ has finite rank and $x_k\wto 0$, we have $\|Rx_k\|\to 0$. Hence
\[
\|y_k-x_k\|=\|Rx_k\|\to 0.
\]
Moreover, $y_k$ has finite support because both $x_k$ and $Rx_k\in E$ do. For $k$ large enough we have $y_k\ne 0$. For such $k$, put
\[
z_k:=\frac{y_k}{\|y_k\|}.
\]
Then $z_k\in M$, $\|z_k\|=1$, $z_k$ has finite support, and $\|z_k-x_k\|\to 0$. Since $x_k$ and $z_k$ are unit vectors, Lemma~\ref{thm:J-ellp}(i) yields
\[
\|J(z_k)-J(x_k)\|\to 0.
\]
Consequently,
\[
\ip{Tz_k}{J(z_k)}\to \lambda.
\]
Also $z_k\wto 0$, because $x_k\wto 0$ and $\|z_k-x_k\|\to 0$. By Lemma~\ref{thm:J-ellp}(ii),
\[
J(z_k)\wto 0
\qquad\text{in }\ell_q.
\]
Since $F$ is finite-dimensional, this implies
\[
\sup\{|\ip{f}{J(z_k)}|: f\in F,\ \|f\|=1\}\to 0.
\]
Choosing $k$ large enough gives the required vector.

\smallskip

\textup{(iii)$\Rightarrow$(i).} Take $F=\{0\}$. Then Theorem~\ref{thm:characterization-We} yields $\lambda\in \We(T)$.
\end{proof}

The next corollary converts the smallness of $J(x)$ on a prescribed finite-dimensional subspace into exact annihilation of that subspace.  This is a finite-dimensional correction of the canonical norming functional supplied by Theorem~\ref{thm:We-ellp-tail-char}.

\begin{corollary}\label{cor:lp-annihilator-We}
Let $T\in\mathcal L(\ell_p)$ and let $\lambda\in\C$. Then
$\lambda\in W_e(T)$ if and only if, for every finite-codimensional subspace
$M\subset \ell_p$, every finite-dimensional subspace $F\subset \ell_p$, and
every $\varepsilon>0$, there are $x\in M$ and $x^*\in F^\perp$ such that
\[
        \|x\|=1,
        \qquad
        \ip{x}{x^*}=1,
        \qquad
        \|x^*\|<1+\varepsilon,
\]
and
\[
        |\ip{Tx}{x^*}-\lambda|<\varepsilon.
\]
In the implication from $\lambda\in W_e(T)$ to this condition, $x$ may moreover
be chosen with finite support.
\end{corollary}

\begin{proof}
Let $R:\ell_q\to F^*$ be the restriction map.  Since $F$ is finite-dimensional, $R$ has a bounded right inverse $L:F^*\to\ell_q$.  Put $C:=\|L\|$.
Choose $0<\delta<1$ so small that
\[
        C\delta<\frac12, \qquad 
        \frac{1+C\delta}{1-C\delta}<1+\varepsilon,
\]
and
\[
        \frac{\delta+\|T\|C\delta}{1-C\delta}
        +\frac{|\lambda|C\delta}{1-C\delta}<\varepsilon .
\]
By Theorem~\ref{thm:We-ellp-tail-char}, choose a unit vector $x\in M$ with finite support such that
\[
        |\ip{Tx}{J(x)}-\lambda|<\delta
        \quad\hbox{and}\quad
        \|RJ(x)\|_{F^*}<\delta .
\]
Define
\[
        c_F:=L(RJ(x)),
        \qquad
        g:=J(x)-c_F .
\]
Then $g\in F^\perp$ and
\(
        \|c_F\|\le C\delta .
\)
Consequently, if
\[
        a:=\ip{x}{g},
\]
then
\[
        |a-1|\le C\delta,
        \qquad
        |a|\ge1-C\delta,
\]
and
\[
        \|g\|\le1+C\delta .
\]
Set
\[
        x^*:=\frac{1}{\overline a}\,g,
\]
and note that 
\[
        x^*\in F^\perp,
        \qquad
        \ip{x}{x^*}=1,
\]
and
\[
        \|x^*\|
        \le\frac{1+C\delta}{1-C\delta}<1+\varepsilon .
\]
Moreover
\[
        |\ip{Tx}{g}-\lambda|
        \le |\ip{Tx}{J(x)}-\lambda|+\|T\|\|c_F\|
        <\delta+\|T\|C\delta .
\]
Since
\[
        \ip{Tx}{x^*}=\frac{1}{a}\ip{Tx}{g},
\]
we get
\[
\begin{aligned}
        |\ip{Tx}{x^*}-\lambda|
        &\le \frac{|\ip{Tx}{g}-\lambda|}{|a|}
        + |\lambda|\left|\frac1a-1\right|  \\
        &\le
        \frac{\delta+\|T\|C\delta}{1-C\delta}
        +\frac{|\lambda|C\delta}{1-C\delta}
        <\varepsilon .
\end{aligned}
\]
This proves the forward implication. Conversely, take $F=\{0\}$ and let the
finite-codimensional subspace $M$ and $\varepsilon>0$ vary. The
approximate-normalisation part of Theorem~\ref{thm:characterization-We} gives
$\lambda\in\We(T)$.
\end{proof}

The next two statements are the $\ell_p$ analogues of a Hilbert-space interior-point construction.  Compare \cite[Corollary~4.3]{MT-circles}, where an orthonormal sequence plays the role taken below by the biorthogonal pairs $(x_n,x_n^*)$.  In the present atomic setting the required finite-dimensional annihilation is supplied by Corollary~\ref{cor:lp-annihilator-We}.

\begin{lemma}\label{lem:interior-We-finite-codim}
Let $T\in \mathcal L(\ell_p)$, and $\lambda\in \Int \We(T)$. Let $F,M\subset \ell_p$ satisfy $\dim F<\infty$ and $\codim M<\infty$. Then, for every $\e>0$, there exists a pair $(x,x^*)\in M\times F^\perp$ such that
\[
\begin{gathered}
\bigl|\|x\|-1\bigr|\le \e,
\qquad
\bigl|\|x^*\|-1\bigr|\le \e,\\
|\ip{x}{x^*}-1|\le \e,
\qquad
\ip{(T-\lambda I)x}{x^*}=0.
\end{gathered}
\]
\end{lemma}

\begin{proof}
Replacing $T$ by $T-\lambda I$, we may assume that $\lambda=0$.  Choose
$r>0$ such that
\[
        \{z\in\C: |z|\le r\}\subset \We(T).
\]
Choose $0<\delta<r/2$ so small that, with
\[
        \gamma:=\sqrt{\frac{2\delta}{r}},
\]
one has
\[
        \gamma<\varepsilon,
        \qquad
        \frac{2\delta}{r}<\varepsilon,
\]
\[
        (1+\delta)(1+\gamma)<1+\varepsilon,
        \qquad
        \frac{1-2\delta/r}{1+\gamma}>1-\varepsilon .
\]
By Corollary~\ref{cor:lp-annihilator-We}, applied to $0\in\We(T)$, choose
$(u,u^*)\in M\times F^\perp$ such that
\[
        \|u\|=1,
        \qquad
        \ip{u}{u^*}=1,
        \qquad
        \|u^*\|<1+\delta,
\]
and, putting
\[
        c:=\ip{Tu}{u^*},
\]
one has
\[
        |c|<\delta .
\]
If $c=0$, the pair $(u,u^*)$ has all the required properties, provided
$\delta<\varepsilon$.  Assume therefore that $c\ne0$ and set
\[
        \eta:=-r\frac{c}{|c|} .
\]
Then $|\eta|=r$, hence $\eta\in\We(T)$.  Define
\[
        M_0:=M\cap\Ker u^*\cap\Ker T^*u^*,
        \qquad
        F_0:=F+\operatorname{span}\{u,Tu\}.
\]
Applying Corollary~\ref{cor:lp-annihilator-We} to $\eta$, $M_0$ and $F_0$, choose
$(v,v^*)\in M_0\times F_0^\perp$ such that
\[
        \|v\|=1,
        \qquad
        \ip{v}{v^*}=1,
        \qquad
        \|v^*\|<1+\delta,
\]
and, with
\[
        d:=\ip{Tv}{v^*},
\]
one has
\[
        |d-\eta|<\delta .
\]
In particular $|d|>r/2$.  Put
\[
        \tau:=-\frac{c}{d} .
\]
Then
\[
        |\tau|\le \frac{2\delta}{r}.
\]
Choose \(a,b\in\mathbb C\) such that
\[
        a\overline b=\tau,
        \qquad
        |a|=|b|=|\tau|^{1/2}\le\gamma.
\]
Finally, set
\[
        x:=u+av,
        \qquad
        x^*:=u^*+bv^* .
\]
Then $x\in M$ and $x^*\in F^\perp$.  The definitions of $M_0$ and $F_0$ give
\[
        \ip{v}{u^*}=0,
        \quad
        \ip{Tv}{u^*}=0,
        \quad
        \ip{u}{v^*}=0,
        \quad
        \ip{Tu}{v^*}=0 .
\]
Therefore
\[
        \ip{Tx}{x^*}=c+a\overline b\,d=c+\tau d=0,
\]
and
\[
        \ip{x}{x^*}=1+\tau .
\]
The norm estimates follow directly from the choice of $\delta$:
\[
        |\|x\|-1|\le |a|\le\gamma<\varepsilon,
\]
and
\[
        \|x^*\|\le (1+\delta)(1+\gamma)<1+\varepsilon .
\]
On the other hand,
\[
        \|x^*\|\ge \frac{|\ip{x}{x^*}|}{\|x\|}
        \ge \frac{1-|\tau|}{1+\gamma}
        \ge \frac{1-2\delta/r}{1+\gamma}>1-\varepsilon .
\]
Also
\[
        |\ip{x}{x^*}-1|=|\tau|<\varepsilon .
\]
Thus $(x,x^*)$ has the required properties in the case $\lambda=0$.
Returning to the original operator gives
\[
        \ip{(T-\lambda I)x}{x^*}=0 .
\]
\end{proof}

Iterating the preceding lemma gives the following biorthogonal consequence.
For a prescribed $\lambda\in\Int W_e(T)$, the result may be viewed as a
$\lambda$-diagonal compression statement, in the spirit of
Theorem~\ref{hilbertwe}.

\begin{corollary}\label{cor:interior-We-biorthogonal}
Let $T\in \mathcal L(\ell_p)$ and $\lambda\in \Int \We(T).$ Then for every $(\e_n)_{n\in\N}\subset (0,\infty)$ there exists a biorthogonal system
\[
(x_n,x_n^*)\in \ell_p\times \ell_q
\qquad (n\in\N)
\]
such that
\[
\|x_n\|=1,
\qquad
\|x_n^*\|<1+\e_n,
\qquad
\ip{x_n}{x_m^*}=\delta_{n,m},
\qquad
\ip{Tx_n}{x_m^*}=\lambda\,\delta_{n,m}
\]
for all $m,n\in\N$.
\end{corollary}

\begin{proof}
For each $n\in\N$, choose $\delta_n>0$ so small that
\[
\frac{(1+\delta_n)^2}{1-\delta_n}<1+\e_n.
\]
We repeat the earlier recursive construction, now using Lemma~\ref{lem:interior-We-finite-codim} as the $\ell_p$ input.  By Lemma~\ref{lem:interior-We-finite-codim}, applied with $\delta_1$, there exist $x_1\in \ell_p$ and $x_1^*\in \ell_q$ such that
\[
\begin{gathered}
|\|x_1\|-1|<\delta_1,
\qquad
|\|x_1^*\|-1|<\delta_1,
\qquad
|\ip{x_1}{x_1^*}-1|<\delta_1,\\
\ip{(T-\lambda I)x_1}{x_1^*}=0.
\end{gathered}
\]
After replacing $(x_1,x_1^*)$ by
\[
\left(\frac{x_1}{\|x_1\|},\frac{\|x_1\|}{\overline{\ip{x_1}{x_1^*}}}x_1^*\right),
\]
we may assume that
\[
\|x_1\|=1=\ip{x_1}{x_1^*},
\qquad
\|x_1^*\|<1+\e_1,
\qquad
\ip{Tx_1}{x_1^*}=\lambda.
\]

Assume that $(x_1,x_1^*),\dots,(x_n,x_n^*)$ have already been constructed so that
\[
\|x_j\|=1,
\qquad
\|x_j^*\|<1+\e_j,
\qquad
\ip{x_j}{x_r^*}=\delta_{j,r},
\qquad
\ip{Tx_j}{x_r^*}=\lambda\,\delta_{j,r}
\]
for $1\le j,r\le n$.  Apply Lemma~\ref{lem:interior-We-finite-codim} with
\[
M:=\bigcap_{j=1}^n \Ker x_j^*\cap \Ker T^*x_j^*,
\qquad
F:=\operatorname{span}\{x_1,\dots,x_n,Tx_1,\dots,Tx_n\}.
\]
We obtain $(u,u^*)\in M\times F^\perp$ such that
\[
\begin{gathered}
|\|u\|-1|<\delta_{n+1},
\qquad
|\|u^*\|-1|<\delta_{n+1},\\
|\ip{u}{u^*}-1|<\delta_{n+1},
\qquad
\ip{(T-\lambda I)u}{u^*}=0.
\end{gathered}
\]
Set
\[
x_{n+1}:=\frac{u}{\|u\|},
\qquad
x_{n+1}^*:=\frac{\|u\|}{\overline{\ip{u}{u^*}}}u^*.
\]
Then
\[
\|x_{n+1}\|=1=\ip{x_{n+1}}{x_{n+1}^*},
\qquad
\|x_{n+1}^*\|
\le
\frac{(1+\delta_{n+1})^2}{1-\delta_{n+1}}
<1+\e_{n+1},
\]
and
\[
\ip{Tx_{n+1}}{x_{n+1}^*}=\lambda.
\]
Since $x_{n+1}\in M$ and $x_{n+1}^*\in F^\perp$, the enlarged family remains biorthogonal and satisfies the required diagonal identities for $T$.  Iterating the procedure completes the proof.
\end{proof}

\subsection{General properties of $\We(T)$ on $\ell_p$}\label{genpropwe}

Theorem~\ref{thm:We-ellp-tail-char} yields in particular convexity of $W_e(T)$ for every 
$T \in \mathcal L(\ell_p),$ thus establishing one of the main results of this paper.

\begin{theorem}\label{thm:We-ellp-convex}
Let $T\in \mathcal L(\ell_p)$. Then $\We(T)$ is a compact convex set.
\end{theorem}

\begin{proof}
Compactness follows from Remark \ref{remark17},
 so only convexity must be shown.
Let $\lambda,\mu\in \We(T)$. It is enough to prove that $(\lambda+\mu)/2\in \We(T)$, since closedness then gives all convex combinations.

Fix $\varepsilon>0$ and $k_0\in\N$. By Theorem~\ref{thm:We-ellp-tail-char}, there exists a unit vector $x\in\ell_p$ with finite support such that
\[
\supp x\subset\{k_0+1,k_0+2,\dots\},
\qquad
|\ip{Tx}{J(x)}-\lambda|<\varepsilon.
\]
Choose $k_1>k_0$ so that
\[
\supp x\subset\{k_0+1,\dots,k_1\},
\qquad
\|(I-P_{k_1})Tx\|<\varepsilon.
\]
Set
\[
M:=(I-P_{k_1})\ell_p\cap \Ker(P_{k_1}T).
\]
Since $(I-P_{k_1})\ell_p$ and $\Ker(P_{k_1}T)$ are both finite-codimensional, so is $M$.
By Theorem~\ref{thm:We-ellp-tail-char}, applied to $M$ and $F=\{0\}$, there exists a unit vector $y\in M$ with finite support such that
\[
|\ip{Ty}{J(y)}-\mu|<\varepsilon.
\]
Because $y\in (I-P_{k_1})\ell_p$, the supports of $x$ and $y$ are disjoint. Moreover,
\[
\ip{Ty}{J(x)}=\ip{P_{k_1}Ty}{J(x)}=0,
\]
while
\[
|\ip{Tx}{J(y)}|=|\ip{(I-P_{k_1})Tx}{J(y)}|\le \|(I-P_{k_1})Tx\|<\varepsilon.
\]
Let
\[
z:=\frac{x+y}{2^{1/p}}.
\]
Since $x$ and $y$ have disjoint supports, $\|z\|=1$ and
\[
J(z)=\frac{J(x)+J(y)}{2^{1/q}}.
\]
Therefore
\begin{align*}
\Bigl|\ip{Tz}{J(z)}-\frac{\lambda+\mu}{2}\Bigr|
&\le \frac12\Bigl(|\ip{Tx}{J(x)}-\lambda|+|\ip{Ty}{J(y)}-\mu|\Bigr)\\
&\qquad +\frac12\Bigl(|\ip{Tx}{J(y)}|+|\ip{Ty}{J(x)}|\Bigr)\\
&<\frac12(\varepsilon+\varepsilon)+\frac12(\varepsilon+0)=\frac{3\varepsilon}{2}.
\end{align*}
Since $k_0$ and $\varepsilon$ were arbitrary, $(\lambda+\mu)/2\in \We(T)$, and hence $\We(T)$ is convex.
\end{proof}

After the tail characterisation and convexity of \(W_e(T)\) for operators on
\(\ell_p\) have been established, we identify \(W_e(T)\) with the algebraic
numerical range of the Calkin image. We first record that, on \(\ell_p\), the
operator quantity used in \cite{BM} gives the usual Calkin norm.

For a Banach space \(X\) and \(S\in\mathcal L(X)\), set
\begin{equation}\label{eq:mu-seminorm}
        \|S\|_\mu
        :=
        \inf\{\|S|_M\|:\ M\subset X,\ \operatorname{codim}M<\infty\}.
\end{equation}
This is one of the standard measures of non-compactness for operators, see,
for example, \cite[Chapter~24]{MullerBook}.

\begin{lemma}\label{lem:mu-essential-lp}
Let \(1<p<\infty\). For every \(S\in\mathcal L(\ell_p)\), one has
\[
        \|S\|_\mu=\|S+\mathcal K(\ell_p)\|.
\]
\end{lemma}

\begin{proof}
Let \(q=p/(p-1)\), and let \(P_n^{(q)}\) denote the coordinate projection
onto \(\operatorname{span}\{f_1,\ldots,f_n\}\) in \(\ell_q\). Combining the
standard identities for operator measures of non-compactness from
\cite[Chapter~24]{MullerBook}, we have
\[
        \|S\|_\mu=\chi(S^*B_{\ell_q}),
\]
where \(B_{\ell_q}\) denotes the closed unit ball of \(\ell_q\), and
\(\chi\) stands for the Hausdorff measure of non-compactness. By the tail
formula for bounded subsets of \(\ell_q\), see
\cite[Theorem~5.18(a)]{BanasMursaleen},
\[
        \chi(S^*B_{\ell_q})
        =
        \lim_{n\to\infty}\|(I-P_n^{(q)})S^*\|.
\]
The last limit is the essential norm of \(S^*\). Indeed, \(P_n^{(q)}S^*\) is
finite-rank, and hence
\[
        \|S^*+\mathcal K(\ell_q)\|
        \le
        \lim_{n\to\infty}\|(I-P_n^{(q)})S^*\|.
\]
Conversely, if \(K\in\mathcal K(\ell_q)\), then
\[
        (I-P_n^{(q)})K\to0
\]
in norm. Therefore
\[
        \limsup_{n\to\infty}\|(I-P_n^{(q)})S^*\|
        \le
        \|S^*-K\|.
\]
Taking the infimum over \(K\in\mathcal K(\ell_q)\) gives the reverse
inequality. Thus
\[
        \|S\|_\mu=\|S^*+\mathcal K(\ell_q)\|.
\]
Since \(\ell_p\) is reflexive, passage to the adjoint preserves the essential
norm. Consequently,
\[
        \|S\|_\mu=\|S+\mathcal K(\ell_p)\|.
\]
\end{proof}

For a Banach space \(X\) and \(T\in\mathcal L(X)\), let \(V_\mu(T)\)
denote the algebraic numerical range of \(\pi(T)\) in
\(\mathcal L(X)/\mathcal K(X)\), equipped with the quotient seminorm
\(\|\cdot\|_\mu\).
It was shown in \cite[Theorem~8]{BM} that
\begin{equation}\label{vmu}
        V_\mu(T)=\overline{\conv}\, W_e(T),
        \qquad T\in\mathcal L(X),
\end{equation}
for every Banach space \(X\). Applying this to
\(X=\ell_p\), and using Lemma~\ref{lem:mu-essential-lp} together with
Theorem~\ref{thm:We-ellp-convex}, we obtain the following Calkin-algebra
identification of the essential numerical range on \(\ell_p\).

\begin{theorem}\label{thm:We-ellp-calkin}
Let $T\in \mathcal L(\ell_p)$. Then
\[
        \We(T)
        =
        V\bigl(\pi(T),\mathcal L(\ell_p)/\mathcal K(\ell_p)\bigr)
\]
(where the Calkin algebra is equipped with the usual quotient norm).
\end{theorem}

\begin{proof}
By \eqref{vmu} and Theorem~\ref{thm:We-ellp-convex},
\[
        V_\mu(T)=\We(T).
\]
By Lemma~\ref{lem:mu-essential-lp}, the quotient seminorm defined by
\(\|\cdot\|_\mu\) on
\(\mathcal L(\ell_p)/\mathcal K(\ell_p)\) coincides with the usual quotient
norm. Hence \(V_\mu(T)\) is precisely
\[
        V\bigl(\pi(T),\mathcal L(\ell_p)/\mathcal K(\ell_p)\bigr),
\]
computed with respect to the usual quotient norm. This proves the assertion.
\end{proof}

Theorem~\ref{thm:We-ellp-calkin} allows us to describe $W_e(T)$ in terms of compact perturbations, extending the corresponding Hilbert-space formula to the whole range $1<p<\infty$.
\begin{theorem}\label{thm:We-ellp-compact-perturb}
Let $T\in \mathcal L(\ell_p)$. Then
\[
\We(T)=\bigcap\{\overline{W(T+K)}: K\in \mathcal K(\ell_p)\}.
\]
\end{theorem}

\begin{proof}
For each compact operator $K$, one has $\overline{W(T+K)}\supset \We(T+K)=\We(T)$. Hence
\[
\bigcap\{\overline{W(T+K)}: K\in \mathcal K(\ell_p)\}\supset \We(T).
\]
On the other hand, by Theorem~\ref{thm:We-ellp-calkin} and Theorem \ref{vrange},
\[
\We(T)=V\bigl(\pi(T),\mathcal L(\ell_p)/\mathcal K(\ell_p)\bigr)
=\bigcap\{V(T+K,\mathcal L(\ell_p)): K\in \mathcal K(\ell_p)\}.
\]
Since by Theorem \ref{vrange} we have $V(T+K,\mathcal L(\ell_p))=\overline{\conv}\, W(T+K)$, 
\[
\We(T)=\bigcap\{\overline{\conv}\, W(T+K): K\in \mathcal K(\ell_p)\}\supset \bigcap\{\overline{W(T+K)}: K\in \mathcal K(\ell_p)\}.
\]
Thus equality holds.
\end{proof}

The compact-perturbation formula in Theorem~\ref{thm:We-ellp-compact-perturb} represents \(\We(T)\) as an intersection of \(\overline{W(T+K)}\) over all compact perturbations \(K\). Under the additional assumption \(\operatorname{Int}\We(T)\ne\varnothing\), it can be strengthened to a single compact perturbation.

\begin{corollary}\label{cor:single-compact-realization}
Let \(1<p<\infty\), and let \(T\in\mathcal L(\ell_p)\). If
\[
        \operatorname{Int}\We(T)\ne\varnothing,
\]
then there exists \(K\in\mathcal K(\ell_p)\) such that
\[
        \We(T)=\overline{W(T+K)}.
\]
\end{corollary}

\begin{proof}
By Theorem~\ref{thm:We-ellp-calkin},
\[
        \We(T)
        =
        V\bigl(\pi(T),\mathcal L(\ell_p)/\mathcal K(\ell_p)\bigr).
\]
Since \(\We(T)\) has non-empty interior, the compact-perturbation
realisation result in \cite[Theorem~3.1, Remark~2]{ChSSW}
yields \(K\in\mathcal K(\ell_p)\) such that
\[
        V(T+K,\mathcal L(\ell_p))
        =
        V\bigl(\pi(T),\mathcal L(\ell_p)/\mathcal K(\ell_p)\bigr).
\]
Hence
\[
        V(T+K,\mathcal L(\ell_p))=\We(T).
\]
On the other hand, in view of Theorem \ref{vrange},
\[
        \We(T)=\We(T+K)\subset \overline{W(T+K)}
        \subset V(T+K,\mathcal L(\ell_p)).
\]
Therefore \(\overline{W(T+K)}=\We(T)\), as claimed.
\end{proof}

Finally, the formula for the numerical range of adjoint $T^*$ in our $\ell_p$ setting is immediate from the definition and the identity $J_q(J_p(x))=x$ for $\|x\|_p=1$ (with $T^*$ understood in the sense of duality fixed at the beginning of this section).
\begin{proposition}\label{prop:We-adjoint-ellp}
Let $T\in \mathcal L(\ell_p)$. Then
\[
\We(T^*)=\{\bar z: z\in \We(T)\}.
\]
\end{proposition}
\begin{remark}
The convexity phenomenon for \(W_e(T)\) is not a general
feature of Banach-space essential numerical ranges. In fact, there exists an
operator on a space isometric to \(\ell_\infty\) whose spatial essential numerical
range is not convex. More precisely, for every \(\eta>0\) one can construct an
operator \(T\in \mathcal L(\ell_\infty)\) such that
\[
        i+\eta,\,-i+\eta\in W_e(T),
        \qquad
        \eta\notin W_e(T).
\]
Thus the midpoint of two points of \(W_e(T)\) may fail to belong to \(W_e(T)\).

The construction starts with an operator on \(\ell_\infty^2\) whose numerical
range has an exposed gap: the values \(i\) and \(-i\) are attained on the
boundary line \(\operatorname{Re} z=0\), while their midpoint \(0\) is not.
This operator is then lifted coordinatewise to
\(\ell_\infty(\mathbb Z;\ell_\infty^2)\). One adds to it a perturbation
supported on a sparse set of coordinates, with coefficients defined by
oscillatory functionals coming from a shift-invariant mean on
\(\ell_\infty(\mathbb Z)\). The finite-codimensional interpolation built into
these functionals makes the two translated endpoint values \(i+\eta\) and
\(-i+\eta\) visible in the spatial essential numerical range. On the other
hand, if the translated midpoint \(\eta\) were approached by numerical-range
values, the Bessel estimate for the oscillatory coefficients would force the
corresponding norming functionals to concentrate asymptotically at a single
sparse coordinate. At that coordinate one is brought back to the original
two-dimensional exposed gap, which excludes the midpoint. Since
\(\ell_\infty(\mathbb Z;\ell_\infty^2)\) is isometric to the scalar space
\(\ell_\infty\), this gives a genuine counterexample on \(\ell_\infty\).

We do not include the details here, since they would lead away from the main
theme of the present paper, which is already rather long. They will be presented
elsewhere. The example shows, however, that the structural assumptions used above
are not merely technical: without some tail or shrinking mechanism, the spatial
essential numerical range may lose convexity even on a classical sequence space.
\end{remark}
\begin{remark}
The appearance of the Calkin algebra in Theorem~\ref{thm:We-ellp-calkin} is also consistent with recent work \cite{Boedihardjo2024}, where it was shown that the \(\ell_p\)-Calkin algebra \(\mathcal L(\ell_p)/\mathcal K(\ell_p)\), \(1<p<\infty\), contains isomorphic copies of all separable \(C^*\)-algebras. This gives another indication that, after passing to the Calkin algebra, part of the Hilbert-space picture survives in the non-Hilbertian \(\ell_p\)-setting.
\end{remark}

\subsection{Interior points and star-shapedness}

Once the structure of $\We(T)$ on $\ell_p$ is in place, one can return to the ordinary numerical range $W(T)$ and ask what geometric information is forced by the essential one. The next results answer this question for interior points and star-shapedness.  We first show that interior points of the essential numerical range yield exact points of the ordinary numerical range.  The proof uses the following tail-vector realisation lemma.

\begin{lemma}\label{lem:interior-halving}
Let $T\in \mathcal L(\ell_p)$, assume that $0\in \Int \We(T)$, and set
\[
\varepsilon:=\dist\bigl(0,\partial \We(T)\bigr).
\]
Let $k\ge 0$, and let $x\in \ell_p$ be a unit vector with finite support contained in $\{k+1,k+2,\dots\}$ such that
\[
\eta:=|\ip{Tx}{J(x)}|\le \varepsilon.
\]
Then there exists a unit vector $y$ with finite support contained in $\{k+1,k+2,\dots\}$ such that
\[
|\ip{Ty}{J(y)}|\le \eta/2
\qquad\text{and}\qquad
\|y-x\|\le 3\Bigl(\frac{\eta}{\varepsilon}\Bigr)^{1/p}.
\]
\end{lemma}

\begin{proof}
If \(\eta=0\), take \(y=x\). Hence assume \(\eta>0\).
Find $k_1>k$ such that
\[
\supp x\subset \{k+1,\dots,k_1\}
\qquad\text{and}\qquad
\|(I-P_{k_1})Tx\|<\eta/4.
\]
Let $c:=\ip{Tx}{J(x)}$, so $|c|=\eta$. Set
\[
c':=-\frac{c\varepsilon}{\eta}.
\]
Then $|c'|=\varepsilon$, hence $c'\in \We(T)$.
Choose a unit vector $u\in \ell_p$ with finite support, satisfying
\[
u\in (I-P_{k_1})\ell_p\cap T^{-1}(I-P_{k_1})\ell_p
\]
and
\[
|\ip{Tu}{J(u)}-c'|<\eta/4.
\]
Now set
\[
y:=\Bigl(\frac{\varepsilon}{\varepsilon+\eta}\Bigr)^{1/p}x+
\Bigl(\frac{\eta}{\varepsilon+\eta}\Bigr)^{1/p}u.
\]
Since $x$ and $u$ have disjoint supports, $\|y\|=1$ and
\[
J(y)=\Bigl(\frac{\varepsilon}{\varepsilon+\eta}\Bigr)^{1/q}J(x)+
\Bigl(\frac{\eta}{\varepsilon+\eta}\Bigr)^{1/q}J(u).
\]
Therefore
\begin{align*}
|\ip{Ty}{J(y)}|
&\le \Bigl|\frac{\varepsilon}{\varepsilon+\eta}c+\frac{\eta}{\varepsilon+\eta}c'\Bigr|+
\frac{\eta}{\varepsilon+\eta}|\ip{Tu}{J(u)}-c'|+|\ip{Tx}{J(u)}|\\
&\le 0+\frac{\eta}{4}+\frac{\eta}{4}=\frac{\eta}{2}.
\end{align*}
Finally,
\[
\|y-x\|\le 1-\Bigl(\frac{\varepsilon}{\varepsilon+\eta}\Bigr)^{1/p}+\Bigl(\frac{\eta}{\varepsilon+\eta}\Bigr)^{1/p}.
\]
Writing
\[
        \alpha:=\frac{\eta}{\varepsilon+\eta},
\]
we have \(0\le \alpha\le 1/2\), and therefore
\[
        \|y-x\|
        \le 1-(1-\alpha)^{1/p}+\alpha^{1/p}.
\]
Moreover, since \(r\mapsto r^{1/p}\) is concave on \([0,\infty)\), we infer that
\(
        (1-\alpha)^{1/p}\ge 1-2\alpha .
\)
Thus
\[
        \|y-x\|
        \le 2\alpha+\alpha^{1/p}
        \le 3\alpha^{1/p}
        \le 3\left(\frac{\eta}{\varepsilon}\right)^{1/p}.
\]
\end{proof}

Now the statement on interior points follows by iterative use of Lemma \ref{lem:interior-halving}.

\begin{theorem}\label{thm:interior-We-inside-W}
Let $T\in \mathcal L(\ell_p)$. Then
\[
\Int \We(T)\subset W(T).
\]
Moreover, for every $\lambda\in \Int \We(T)$ and every $k\ge 0$, there exists a unit vector $x$ with
\[
\supp x\subset \{k+1,k+2,\dots\}
\qquad\text{and}\qquad
\ip{Tx}{J(x)}=\lambda.
\]
\end{theorem}

\begin{proof}
Let $\lambda\in \Int \We(T)$ and $k\ge 0$. Without loss of generality we may assume that $\lambda=0$.
Choose $\varepsilon>0$ such that
\[
\{z\in\C: |z|\le \varepsilon\}\subset \We(T).
\]
Choose a unit vector $x_0\in \ell_p$ with finite support contained in $\{k+1,k+2,\dots\}$ and satisfying
\[
|\ip{Tx_0}{J(x_0)}|\le \varepsilon.
\]
Using Lemma~\ref{lem:interior-halving} inductively, construct unit vectors $x_n\in \ell_p$ with finite support in $\{k+1,k+2,\dots\}$ such that
\[
|\ip{Tx_n}{J(x_n)}|\le \frac{\varepsilon}{2^n}
\]
and
\[
\|x_n-x_{n-1}\|\le 3\cdot 2^{(-n+1)/p}.
\]
Thus $(x_n)$ is a Cauchy sequence. Let $x=\lim_n x_n$. Then $\|x\|=1$, the support of $x$ is still contained in $\{k+1,k+2,\dots\}$, and Lemma~\ref{thm:J-ellp}(i) gives $J(x_n)\to J(x)$. Hence
\[
\ip{Tx}{J(x)}=\lim_n \ip{Tx_n}{J(x_n)}=0.
\]
This proves the theorem.
\end{proof}

While the interior of \(W_e(T)\) produces points in \(W(T)\), the whole of
\(W_e(T)\) provides star-centres for \(\overline{W(T)}\),  thus in particular making $\overline{W(T)}$ star-shaped.

\begin{theorem}\label{thm:closureW-starshaped}
Let $T\in \mathcal L(\ell_p)$. Then $\overline{W(T)}$ is star-shaped, and every point of $\We(T)$ is a star-centre of $\overline{W(T)}$.
\end{theorem}

\begin{proof}
Let $\mu\in \We(T)$, $\lambda\in \overline{W(T)}$, $\varepsilon>0$, and $t\in(0,1)$.
Choose a unit vector $x\in \ell_p$ such that
\[
|\ip{Tx}{J(x)}-\lambda|<\varepsilon.
\]
For $k\in\N$, set
\[
x_k:=\frac{P_kx}{\|P_kx\|}.
\]
For $k$ large enough, $x_k$ is defined, $x_k\to x$ in norm, and hence $J(x_k)\to J(x)$ by Lemma~\ref{thm:J-ellp}(i). Thus there exists $k_0\in\N$ such that
\[
|\ip{Tx_{k_0}}{J(x_{k_0})}-\lambda|<\varepsilon,
\qquad
\|(I-P_{k_0})Tx_{k_0}\|<\varepsilon.
\]
Write $x':=x_{k_0}$. Since $\mu\in \We(T)$, Theorem~\ref{thm:We-ellp-tail-char}\textup{(iii)}, applied to
\[
M:=(I-P_{k_0})\ell_p\cap \Ker(P_{k_0}T)
\qquad\text{and}\qquad F=\{0\},
\]
yields a unit vector $y\in M$ with finite support such that
\[
|\ip{Ty}{J(y)}-\mu|<\varepsilon.
\]
In particular,
\[
\supp y\subset \{k_0+1,k_0+2,\dots\}
\qquad\text{and}\qquad
P_{k_0}Ty=0.
\]
Set
\[
z:=t^{1/p}x'+(1-t)^{1/p}y.
\]
Then $\|z\|=1$ and
\[
J(z)=t^{1/q}J(x')+(1-t)^{1/q}J(y).
\]
Therefore
\begin{align*}
\bigl|\ip{Tz}{J(z)}-(t\lambda+(1-t)\mu)\bigr|
&\le t|\ip{Tx'}{J(x')}-\lambda|+(1-t)|\ip{Ty}{J(y)}-\mu|\\
&\qquad +|\ip{Tx'}{J(y)}|+|\ip{Ty}{J(x')}|\\
&\le t\varepsilon+(1-t)\varepsilon+\varepsilon+0=2\varepsilon,
\end{align*}
where
\[
|\ip{Tx'}{J(y)}|=|\ip{(I-P_{k_0})Tx'}{J(y)}|\le \|(I-P_{k_0})Tx'\|<\varepsilon,
\]
while
\[
\ip{Ty}{J(x')}=\ip{P_{k_0}Ty}{J(x')}=0.
\]
Since $\varepsilon>0$ was arbitrary, $t\lambda+(1-t)\mu\in \overline{W(T)}$.
\end{proof}

If $\Int \We(T)\ne \varnothing$, one can strengthen the preceding result from $\overline{W(T)}$ to $W(T),$
  thus showing that on \(\ell_p\) the numerical range is often star-shaped. To this end, we need the next auxiliary lemma.

\begin{lemma}\label{lem:preiterate}
Let \(T\in \mathcal L(\ell_p)\), let
\(0\in \operatorname{Int} W_e(T)\), and choose \(\varepsilon>0\) such that
\[
        \{z\in\mathbb C: |z|\le \varepsilon\}\subset W_e(T).
\]
Let \(0<s'<s\le 2s'\), and suppose that \(x\in\ell_p\) is a unit vector
with finite support satisfying
\[
        |\langle Tx,J(x)\rangle-s|
        \le
        \frac{(s-s')\varepsilon}{s'} .
\]
Then, for every \(\delta>0\), there exists a unit vector \(y\in\ell_p\) with
finite support such that
\[
        |\langle Ty,J(y)\rangle-s'|<\delta
\]
and
\[
        \|y-x\|
        \le
        3\left(\frac{s-s'}s\right)^{1/p}.
\]
\end{lemma}
\begin{proof}
Set
\[
        c':=
        \frac{s'\bigl(s-\langle Tx,J(x)\rangle\bigr)}{s-s'} .
\]
By the hypothesis,
\[
        |c'|
        \le
        \frac{s'}{s-s'}|\langle Tx,J(x)\rangle-s|
        \le \varepsilon .
\]
Hence \(c'\in W_e(T)\).

Choose \(k_1\) so large that
\[
        \operatorname{supp}x\subset\{1,\ldots,k_1\},
        \qquad
        \|(I-P_{k_1})Tx\|<\delta/2 .
\]
By Theorem~\ref{thm:We-ellp-tail-char}, choose a unit vector \(u\in\ell_p\) with finite support such that
\[
        u\in (I-P_{k_1})\ell_p\cap T^{-1}(I-P_{k_1})\ell_p
\]
and
\[
        |\langle Tu,J(u)\rangle-c'|<\delta/2 .
\]
Set
\[
        y:=
        \left(\frac{s'}s\right)^{1/p}x+
        \left(\frac{s-s'}s\right)^{1/p}u .
\]
Then \(\|y\|=1\), \(y\) has finite support, and
\[
        J(y)=
        \left(\frac{s'}s\right)^{1/q}J(x)+
        \left(\frac{s-s'}s\right)^{1/q}J(u).
\]
Since \(Tu\in (I-P_{k_1})\ell_p\) and \(J(x)\) is supported in
\(\{1,\ldots,k_1\}\), we have
\[
        \langle Tu,J(x)\rangle=0.
\]
Moreover,
\[
        |\langle Tx,J(u)\rangle|
        =
        |\langle (I-P_{k_1})Tx,J(u)\rangle|
        <\delta/2 .
\]
Therefore
\[
\begin{aligned}
        |\langle Ty,J(y)\rangle-s'|
        &\le
        \left|
        \frac{s'}s\langle Tx,J(x)\rangle
        +\frac{s-s'}s c'-s'
        \right|                                      \\
        &\quad+
        \frac{s-s'}s
        |\langle Tu,J(u)\rangle-c'|
        +
        |\langle Tx,J(u)\rangle| .
\end{aligned}
\]
The first term is zero by the definition of \(c'\). Hence
\[
        |\langle Ty,J(y)\rangle-s'|<\delta .
\]

Finally,
\[
        \|y-x\|
        \le
        1-\left(\frac{s'}s\right)^{1/p}
        +
        \left(\frac{s-s'}s\right)^{1/p}.
\]
Since \(s\le 2s'\), the number
\[
        \alpha:=\frac{s-s'}s
\]
belongs to \([0,1/2]\). Then, arguing as in the end of the proof of Lemma~\ref{lem:interior-halving}, we obtain
\[
        \|y-x\|
        \le
        3\left(\frac{s-s'}s\right)^{1/p}.
\]
\end{proof}

Now using the above lemma iteratively, we deduce star-shapedness of $W(T)$ under very mild assumptions.

\begin{theorem}\label{thm:W-starshaped-when-intWe}
Let \(T\in\mathcal L(\ell_p)\). If $\Int \We(T)\ne \varnothing$, then $W(T)$ is star-shaped. More precisely, every point of $\Int \We(T)$ is a star-centre of $W(T)$.
\end{theorem}
\begin{proof}
Let \(\lambda\in \operatorname{Int}W_e(T)\), let \(\mu\in W(T)\), and let
\(t\in(0,1)\). We show that
\[
        (1-t)\lambda+t\mu\in W(T).
\]
If \(\lambda=\mu\), there is nothing to prove. Otherwise, replacing \(T\) by
\[
        \frac{T-\lambda I}{\mu-\lambda},
\]
we may assume that
\[
        \lambda=0,\qquad \mu=1.
\]
Let
\[
        \varepsilon:=\operatorname{dist}(0,\partial W_e(T)).
\]
For \(n\ge0\), set
\[
        s_n:=t+\frac{1-t}{2^n}.
\]
Then \(s_0=1\) and \(s_n\downarrow t\).

Choose a unit vector \(u\in\ell_p\) such that
\[
        \langle Tu,J(u)\rangle=1.
\]
For \(m\) large enough, the normalised truncations
\[
        u_m:=\frac{P_m u}{\|P_m u\|}
\]
are well defined and satisfy
\[
        \langle Tu_m,J(u_m)\rangle\to 1
\]
by Lemma~\ref{thm:J-ellp}(i). Hence we may choose a unit vector \(x_0\in\ell_p\) with
finite support such that
\[
        |\langle Tx_0,J(x_0)\rangle-s_0|
        <
        \frac{(1-s_1)\varepsilon}{s_1}.
\]

We construct inductively unit vectors \(x_n\in\ell_p, n \in \mathbb N,\) with finite support
such that
\begin{equation}\label{ind1}
        |\langle Tx_n,J(x_n)\rangle-s_n|
        <
        \frac{(1-s_1)\varepsilon}{2^n s_1}
\end{equation}
and
\begin{equation}\label{ind2}
        \|x_n-x_{n-1}\|
        \le
        3\left(\frac{s_{n-1}-s_n}{s_{n-1}}\right)^{1/p}
        \qquad (n\ge1).
\end{equation}
Suppose that \(x_{n-1}\) has already been constructed. Since \(s_n\le s_1\),
we have
\[
        \frac{(1-s_1)\varepsilon}{2^{n-1}s_1}
        =
        \frac{(1-t)\varepsilon}{2^n s_1}
        \le
        \frac{(1-t)\varepsilon}{2^n s_n}
        =
        \frac{(s_{n-1}-s_n)\varepsilon}{s_n}.
\]
Thus the hypothesis of Lemma~\ref{lem:preiterate} is satisfied with \(s=s_{n-1}\) and
\(s'=s_n\). Applying Lemma~\ref{lem:preiterate} with
\[
        \delta:=\frac{(1-s_1)\varepsilon}{2^n s_1}
\]
gives a unit vector \(x_n\in\ell_p\) with finite support such that
\eqref{ind1} and \eqref{ind2} hold.
This completes the induction step.

Since \(s_{n-1}\ge t\) and
\[
        s_{n-1}-s_n=\frac{1-t}{2^n},
\]
we have
\[
        \|x_n-x_{n-1}\|
        \le
        3\left(\frac{1-t}{2^n t}\right)^{1/p}.
\]
The right-hand side is summable in \(n\), and hence \((x_n)\) is a Cauchy
sequence. Let
\[
        x:=\lim_{n\to\infty}x_n .
\]
Then \(\|x\|=1\), and Lemma~\ref{thm:J-ellp}(i) gives \(J(x_n)\to J(x)\) in \(\ell_q\).
Since
\[
        s_n\to t
        \quad\text{and}\quad
        \frac{(1-s_1)\varepsilon}{2^n s_1}\to0,
\]
we obtain
\[
        \langle Tx,J(x)\rangle
        =
        \lim_{n\to\infty}\langle Tx_n,J(x_n)\rangle
        =
        t.
\]
Thus \(t\in W(T)\) for the normalised operator. If \(T_0\) denotes the
operator before the affine reduction, the preceding conclusion applied to
\((T_0-\lambda I)/(\mu-\lambda)\) gives a unit vector \(x\) such that
\[
        \left\langle
        \frac{T_0-\lambda I}{\mu-\lambda}x,J(x)
        \right\rangle=t.
\]
Equivalently,
\[
        \langle T_0x,J(x)\rangle=(1-t)\lambda+t\mu,
\]
and hence \((1-t)\lambda+t\mu\in W(T_0)\).
Therefore every point of \(\operatorname{Int}W_e(T)\) is a star-centre of
\(W(T)\), and \(W(T)\) is star-shaped.
\end{proof}

The structure theory developed in this section suggests that concrete matrix examples on sequence spaces should display both Hilbert-space features and genuinely non-Hilbertian behaviour. We therefore turn next to Toeplitz-type operators and to the discrete Hilbert transform.

\section{Discrete Toeplitz and related sequence-space examples}\label{secexamp}

The atomic theory on $\ell_p$ already yields explicit sequence-space examples which serve as a natural bridge to the later $L^p$ questions. We record here only those parts of the discrete Toeplitz discussion that are used below and that fit the present scope. Topological properties of $W(T)$ and $W_e(T)$ are invariant under isometric conjugacy, so one may pass freely between isometric examples such as $\ell_p(\mathbb N)$ and $\ell_p(\mathbb N_0)$, where $\mathbb N_0=\mathbb N \cup \{0\}$, or between different enumerations of $\ell_p(\mathbb Z)$.
In particular, we use \(\mathbb N_0\) in the one-sided shift examples only for
notational convenience; of course \(\ell_p(\mathbb N_0)\) is canonically
isometric to \(\ell_p(\mathbb N)\).
 By contrast, the operator classes discussed below have natural domains of their own, and it is convenient to state the results directly for those examples.

\subsection{The unilateral sequence-space example}\label{sec:unilateral-sequence-model}

We begin the example section with one-sided sequence-space Toeplitz operators. As above, let
$1<p<\infty$, and let $(e_j)_{j\ge0}$ be the standard basis of
$\ell_p(\mathbb N_0)$. We denote by $S_R$ and $S_L$ the right and left shifts on
$\ell_p(\mathbb N_0)$, given by
\begin{equation}\label{shifts}
S_Re_j=e_{j+1}\quad(j\ge0),\qquad
S_Le_0=0,\qquad S_Le_j=e_{j-1}\quad(j\ge1).
\end{equation}
We continue to write $J$ for the duality map on $\ell_p$. These operators provide the simplest discrete analogues of Hardy-space Toeplitz operators and already exhibit the main essential-range mechanism.
For operator-algebraic aspects of the \(\ell_p\)-Toeplitz algebra on
\(\ell_p(\mathbb N_0)\), related to $S_R$ and $S_L$, see e.g. \cite{Wang2021}.
Here we use only the concrete Toeplitz examples as test objects for ordinary
and essential numerical ranges.

Sequence-space Toeplitz operators are described primarily by their coefficient sequences. Thus a one-sided Toeplitz operator on $\ell_p(\mathbb N_0)$ is given by a matrix $\bigl(a_{j-k}\bigr)_{j,k\ge0}$ which is constant along diagonals, while the corresponding bilateral Laurent operator on $\ell_p(\mathbb Z)$ is given by convolution with the same coefficient sequence. When the coefficients arise as Fourier coefficients of a bounded function $\varphi$ on $\mathbb T$, it is natural to regard $\varphi$ as a symbol. In the Hilbert-space Hardy setting, by contrast, analytic symbols have a direct numerical-range interpretation. For $\varphi\in H^\infty(\mathbb D)$, let $T_\varphi$ be the Toeplitz operator on $H^2(\mathbb D)$, given by $T_\varphi f=P_+(\varphi f)$, where $P_+$ denotes the Riesz projection from $L^2(\mathbb T)$ onto $H^2(\mathbb D)$. Then 
\[
        \overline{W(T_\varphi)}=\overline{\conv}\,\varphi(\mathbb D)
\]
by \cite{KleinToeplitz}. The discussion below retains the Toeplitz pattern on $\ell_p(\mathbb N_0)$ while keeping the geometry completely explicit and illustrating how differently numerical ranges behave outside the Hilbert setting.

Now let $a=(a_k)_{k\in\mathbb Z}\in \ell^1(\mathbb Z)$ and define the unilateral Toeplitz operator $T_a$ on $\ell_p(\mathbb N_0)$
by
\begin{equation}\label{toepldef}
(T_a x)_n:=\sum_{m\ge 0} a_{n-m}x_m,
\qquad n\ge 0.
\end{equation}

The next simple proposition records a special feature of Toeplitz numerical
ranges on \(\ell_p(\mathbb N_0)\), \(1<p<\infty\), resembling the one for \(p=2\).
\begin{proposition}\label{prop:toeplitz-shift-lemma}
For every $T_a \in \mathcal L(\ell_p(\mathbb N_0))$ defined by \eqref{toepldef} one has
\[
\We(T_a)=\overline{W(T_a)}.
\]
As a consequence,
$\overline{W(T_a)}$ is a convex set.
\end{proposition}

\begin{proof}
For every unit vector \(x\in \ell_p(\mathbb N_0)\) and every \(m\ge0\),
the corresponding numerical-range value is invariant under the right shift
\(S_R^m\). Indeed, for \(n\ge m\) one has
\[
        (T_a(S_R^m x))_n=(T_ax)_{n-m},
\]
whereas \((S_R^m x)_n=0\) for \(n<m\). Moreover,
\(J(S_R^m x)=S_R^mJ(x)\). Therefore
\[
\begin{aligned}
\ip{T_a(S_R^m x)}{J(S_R^m x)}
&=\sum_{n\ge m} (T_a(S_R^m x))_n\, (J(S_R^m x))_n  \\
&=\sum_{n\ge m} (T_ax)_{n-m}\, J(x)_{n-m} \\
&=\sum_{k\ge0} (T_ax)_k\,J(x)_k
=\ip{T_ax}{J(x)}.
\end{aligned}
\]

Since $S_R^m x\wto 0$, this yields
\[
W(T_a)\subset \We(T_a).
\]
The reverse inclusion
\[
\We(T_a)\subset \overline{W(T_a)}
\]
is tautological from the definition of $\We(T_a)$. Hence
\[
\We(T_a)=\overline{W(T_a)}.
\]
\end{proof}

\subsection{Explicit tridiagonal Toeplitz examples on $\ell_p(\mathbb N_0)$}

We next specialise the general Toeplitz discussion to tridiagonal examples. This is the first place where one can write down explicit formulas and compare upper bounds with exact descriptions. The first theorem gives a general enclosure for $W_e(T_a)=\overline{W(T_a)}.$

Let
\begin{equation}\label{toeplitz}
T=a_{-1}S_L+a_0I+a_1S_R\in \mathcal L(\ell_p(\mathbb N_0)),
\qquad a_{-1},a_0,a_1\in\C,
\end{equation}
where $S_R$ and $S_L$ are given by \eqref{shifts}.
Define
\[
\mathcal E_p(a_1,a_{-1})
:=\left\{
\begin{array}{c}
a_1rs^{p-1}e^{-i\theta}+a_{-1}sr^{p-1}e^{i\theta}:\\[1mm]
r,s\ge0,\ \ r^p+s^p=1,\ \ \theta\in[0,2\pi]
\end{array}
\right\}.
\]

For fixed $r,s\ge0$ with $r^p+s^p=1$, the curve in the parameter $\theta$
which occurs in the definition of $\mathcal E_p(a_1,a_{-1})$ is an ellipse,
possibly degenerate.

\begin{theorem}
\label{thm:toeplitz-tridiag-upper}
Let $T\in \mathcal L(\ell_p(\mathbb N_0))$ be given by \eqref{toeplitz}.
Then
\[
\We(T)=\overline{W(T)}
\subset a_0+2\,\overline{\conv}\,\mathcal E_p(a_1,a_{-1}).
\]
\end{theorem}

\begin{proof}
Let $x\in \ell_p(\mathbb N_0)$ with $\|x\|_p=1.$
 A direct expansion gives
\begin{align*}
\ip{Tx}{J(x)}
&=a_0\sum_{n\ge0}|x_n|^p
+a_1\sum_{n\ge1}x_{n-1}|x_n|^{p-2}\overline{x_n}
+a_{-1}\sum_{n\ge0}x_{n+1}|x_n|^{p-2}\overline{x_n}\\
&=a_0+\sum_{n\ge0}\Bigl(a_1x_n|x_{n+1}|^{p-2}\overline{x_{n+1}}+a_{-1}x_{n+1}|x_n|^{p-2}\overline{x_n}\Bigr).
\end{align*}
For each $n\ge0$, write
\[
        x_n=\rho_ne^{i\varphi_n},
        \qquad \rho_n\ge0,\quad \varphi_n\in[0,2\pi),
\]
and set $d_n:=\rho_n^p+\rho_{n+1}^p$. If $d_n>0$, define
\[
r_n:=\rho_n/d_n^{1/p},\qquad s_n:=\rho_{n+1}/d_n^{1/p},\qquad \theta_n:=\varphi_{n+1}-\varphi_n.
\]
Then $r_n^p+s_n^p=1$, and the $n$th bracket equals
\[
d_n\Bigl(a_1r_ns_n^{p-1}e^{-i\theta_n}+a_{-1}s_nr_n^{p-1}e^{i\theta_n}\Bigr)
\in d_n\,\mathcal E_p(a_1,a_{-1}).
\]
Since
\[
\sum_{n\ge0}d_n=\sum_{n\ge0}(\rho_n^p+\rho_{n+1}^p)=2-\rho_0^p\le 2,
\]
and $\mathcal E_p(a_1,a_{-1})$ is centrally symmetric (replace $\theta$ by $\theta+\pi$), one has
\[
0\in \conv\bigl(\mathcal E_p(a_1,a_{-1})\bigr).
\]
Let
\[
C:=\conv\bigl(\mathcal E_p(a_1,a_{-1})\bigr).
\]
For each $n$ with $d_n>0$, choose $\eta_n\in\mathcal E_p(a_1,a_{-1})$ so that the $n$th bracket equals $d_n\eta_n$. Since $0\in C$ and $\sum_{n\ge0}d_n\le2$, we have
\[
\frac12\sum_{n\ge0}d_n\eta_n
=\sum_{n\ge0}\frac{d_n}{2}\eta_n+\left(1-\frac12\sum_{n\ge0}d_n\right)0\in C.
\]
Thus
\[
\ip{Tx}{J(x)}-a_0=\sum_{n\ge0}d_n\eta_n\in 2C
=2\,\conv\bigl(\mathcal E_p(a_1,a_{-1})\bigr).
\]
Taking closures yields
\[
\overline{W(T)}\subset a_0+2\,\overline{\conv}\,\mathcal E_p(a_1,a_{-1}).
\]
Now apply Proposition~\ref{prop:toeplitz-shift-lemma}.
\end{proof}

The above inclusion can be complemented as follows.
\begin{proposition}\label{prop:toeplitz-spectral-ellipse}
Let $T \in \mathcal L(\ell_p(\mathbb N_0))$ be defined by \eqref{toeplitz},
and put 
\[
        \Gamma_T:=a_0+\{a_1e^{-it}+a_{-1}e^{it}:0\le t\le2\pi\},
        \qquad
        \Delta_T:=\conv\Gamma_T .
\]
Then
\[
        \sigma_e(T)=\Gamma_T,
        \qquad
        \Delta_T\subset \We(T)=\overline{W(T)}.
\]
Moreover, the set $\Delta_T$ has the following explicit form. If
$a_1=a_{-1}=0$, then $\Delta_T=\{a_0\}$. If
$|a_1|\ne |a_{-1}|$, then $\Delta_T$ is the closed ellipse centred at
$a_0$ with 
semiaxis lengths
\[
        |a_1|+|a_{-1}|,
        \qquad
        \bigl||a_{-1}|-|a_1|\bigr|,
\]
with the usual interpretation as a disc when one of $a_1,a_{-1}$ is zero.
If $|a_1|=|a_{-1}|=\rho>0$, then $\Delta_T$ is a line segment centred at
$a_0$ and of length $4\rho$.
\end{proposition}

\begin{proof}
The case $a_1=a_{-1}=0$ is immediate, so suppose that at least one of
$a_1,a_{-1}$ is non-zero. Choose arguments $\phi_1$ and $\phi_{-1}$ so
that
\[
        a_1=|a_1|e^{i\phi_1},
        \qquad
        a_{-1}=|a_{-1}|e^{i\phi_{-1}},
\]
choosing the argument of a zero coefficient arbitrarily, and put
\[
        \omega=e^{i(\phi_1+\phi_{-1})/2}.
\]
Then
\[
        a_1e^{-it}+a_{-1}e^{it}
        =\omega\Bigl((|a_1|+|a_{-1}|)\cos\eta
        +i(|a_{-1}|-|a_1|)\sin\eta\Bigr),
\]
where
\[
        \eta=t+\frac{\phi_{-1}-\phi_1}{2}.
\]
This gives the asserted geometric description of $\Delta_T$.
The essential-spectrum identity $\sigma_e(T)=\Gamma_T$ is the classical
fact from the theory of finite-band Toeplitz operators on $\ell_p.$ See e.g.  \cite[Theorem~1]{Duren1964}.
The lower inclusion follows from Theorem~\ref{thm:convhull-sigma-pie} and
Proposition~\ref{prop:toeplitz-shift-lemma}.
\end{proof}

\begin{remark}
The ellipse $\Delta_T$ is already visible in the upper estimate of
Theorem~\ref{thm:toeplitz-tridiag-upper}. Indeed, taking
$r=s=2^{-1/p}$ in the definition of $\mathcal E_p(a_1,a_{-1})$ gives
\[
        2\bigl(a_1rs^{p-1}e^{-i\theta}+a_{-1}sr^{p-1}e^{i\theta}\bigr)
        =a_1e^{-i\theta}+a_{-1}e^{i\theta},
\]
and hence
\[
        \Delta_T\subset a_0+2\conv\bigl(\mathcal E_p(a_1,a_{-1})\bigr).
\]
\end{remark}
In the symmetric case, $W_e(T_a)$ admits an exact description. The arguments rest on convexity of $W_e(T)$
and the properties of $W_e(T_a)$ proved above.
\begin{theorem}
\label{thm:toeplitz-symmetric-exact}
Let $T \in \mathcal L(\ell_p(\mathbb N_0))$ be given by
\[
T:=a_0I+c(S_R+S_L)\in \mathcal L(\ell_p(\mathbb N_0)),
\qquad a_0,c\in\C.
\]
Define
\[
E_p:=\Bigl\{rs^{p-1}e^{-i\theta}+sr^{p-1}e^{i\theta}:\ r,s\ge0,\ r^p+s^p=1,\ \theta\in[0,2\pi]\Bigr\}.
\]
Then
\[
\We(T)=\overline{W(T)}=a_0+2c\,\overline{\conv}\,E_p.
\]
In particular, for $T_1:=S_R+S_L$, one has
\[
\We(T_1)=\overline{W(T_1)}=2\,\overline{\conv}\,E_p.
\]
\end{theorem}

\begin{proof}
The upper inclusion follows directly from Theorem~\ref{thm:toeplitz-tridiag-upper} with
\[
a_{-1}=a_1=c.
\]
Thus it remains to prove
\[
a_0+2c\,\overline{\conv}\,E_p\subset \We(T).
\]
By Theorem~\ref{thm:We-ellp-convex} and Proposition~\ref{prop:toeplitz-shift-lemma}, the set $\We(T)=\overline{W(T)}$ is closed and convex.
It is therefore enough to prove that
\[
a_0+2c\eta_p\in \We(T)
\qquad (\eta_p\in E_p).
\]
Fix $\eta_p\in E_p$. By definition, there exist $r,s\ge0$ and $\theta\in[0,2\pi]$ such that
\[
r^p+s^p=1,
\qquad
 \eta_p=rs^{p-1}e^{-i\theta}+sr^{p-1}e^{i\theta}.
\]
Set
\[
\alpha:=r,
\qquad
\beta:=se^{i\theta}.
\]
For each $M\in\mathbb N$, define $x^{(M)}\in\ell_p(\mathbb N_0)$ by
\[
x^{(M)}:=M^{-1/p}(\alpha,\beta,\alpha,\beta,\dots,\alpha,\beta,0,0,\dots),
\]
where the alternating block $(\alpha,\beta)$ is repeated $M$ times. Then
\[
\|x^{(M)}\|_p^p
=M^{-1}\,M\,(|\alpha|^p+|\beta|^p)=r^p+s^p=1,
\]
so $\|x^{(M)}\|_p=1$.
Now
\begin{equation}\label{ep}
\ip{Tx^{(M)}}{J(x^{(M)})}
=a_0+c\sum_{n=0}^{2M-2}
\Bigl(x^{(M)}_n|x^{(M)}_{n+1}|^{p-2}\overline{x^{(M)}_{n+1}}
+x^{(M)}_{n+1}|x^{(M)}_n|^{p-2}\overline{x^{(M)}_n}\Bigr).
\end{equation}
For each $n\in\{0,\dots,2M-2\}$, the pair $(x^{(M)}_n,x^{(M)}_{n+1})$ is either
\[
(M^{-1/p}\alpha,M^{-1/p}\beta)
\qquad\text{or}\qquad
(M^{-1/p}\beta,M^{-1/p}\alpha).
\]
In both cases the corresponding bracket equals $M^{-1}\eta_p$. Hence every one of the $2M-1$ brackets in \eqref{ep} contributes exactly $M^{-1}\eta_p$, and therefore
\[
\ip{Tx^{(M)}}{J(x^{(M)})}
=a_0+c\,\frac{2M-1}{M}\,\eta_p
\xrightarrow[M\to\infty]{}
a_0+2c\eta_p.
\]
Thus $a_0+2c\eta_p\in\overline{W(T)}=\We(T)$. Since $\eta_p\in E_p$ was arbitrary, we obtain
\[
a_0+2cE_p\subset \We(T).
\]
By convexity and closedness of $\We(T)$, this gives
\[
a_0+2c\,\overline{\conv}\,E_p\subset \We(T).
\]
Combining this with the upper inclusion completes the proof.
\end{proof}

\begin{corollary}
\label{cor:symmetric-toeplitz-bounds}
For \(T_1=S_R+S_L\), as in Theorem~\ref{thm:toeplitz-symmetric-exact}, one has
\[
\sup\{\operatorname{Re} z:z\in\We(T_1)\}=2,
\qquad
\sup\{|\operatorname{Im} z|:z\in\We(T_1)\}=2\beta_p,
\]
where
\[
\beta_p:=\max\{|sr^{p-1}-rs^{p-1}|:r,s\ge0,\ r^p+s^p=1\}.
\]
\end{corollary}

\begin{proof}
By Theorem~\ref{thm:toeplitz-symmetric-exact},
\[
\We(T_1)=2\,\overline{\conv}\,E_p.
\]
For fixed $r,s$, varying $\theta$ gives
\[
rs^{p-1}e^{-i\theta}+sr^{p-1}e^{i\theta}
=\bigl(rs^{p-1}+sr^{p-1}\bigr)\cos\theta
+i\bigl(sr^{p-1}-rs^{p-1}\bigr)\sin\theta.
\]
Thus the corresponding ellipse, possibly degenerate, has semiaxes
\[
rs^{p-1}+sr^{p-1}
\qquad\text{and}\qquad
|sr^{p-1}-rs^{p-1}|.
\]
By H\"older's inequality,
\[
rs^{p-1}+sr^{p-1}\le (r^p+s^p)^{1/p}(s^p+r^p)^{1/q}=1,
\]
and equality is attained at $r=s=2^{-1/p}$. This gives
\[
\sup\{\operatorname{Re} z:z\in\We(T_1)\}=2.
\]
The formula for the maximal imaginary part is immediate from the same ellipse description and the definition of $\beta_p$.
\end{proof}

\begin{remark}\label{rem:symmetric-spectral-comparison}
For $T_1$, Proposition~\ref{prop:toeplitz-spectral-ellipse} gives
\[
        \conv\sigma(T_1)=\conv\sigma_e(T_1)=[-2,2].
\]
Corollary~\ref{cor:symmetric-toeplitz-bounds} shows that the exact essential numerical range keeps this horizontal semiaxis equal to $2$, but has transverse semiaxis $2\beta_p$. Thus for $p=2$ the set collapses to the spectral segment, while for $p\ne2$ the numerical range is a genuine two-dimensional thickening of that segment. The phase reduction in Remark~\ref{rem:toeplitz-balanced-phase-reduction} gives the same conclusion, after rotation and translation, whenever $|a_{-1}|=|a_1|$.
\end{remark}

\begin{remark}
\label{rem:toeplitz-balanced-phase-reduction}
The same exact formula extends to
\[
T=a_{-1}S_L+a_0I+a_1S_R
\qquad\text{whenever } |a_{-1}|=|a_1|.
\]
Indeed, if $a_{-1}=\rho e^{i\beta}$ and $a_1=\rho e^{i\alpha}$, choose
\[
\omega:=e^{i(\alpha-\beta)/2}
\qquad (|\omega|=1),
\]
and let $D:\ell_p(\mathbb N_0)\to\ell_p(\mathbb N_0)$ be the diagonal surjective isometry given by
\[
(Dx)_n:=\omega^n x_n, \qquad (n \ge 0).
\]
Then
\[
D^{-1}S_RD=\overline\omega\,S_R,
\qquad
D^{-1}S_LD=\omega\,S_L,
\]
so
\[
D^{-1}TD=a_0I+c(S_R+S_L),
\qquad
c:=\rho e^{i(\alpha+\beta)/2}.
\]
Since both the numerical range and the essential numerical range are invariant under the mapping \(T \to D^{-1}TD\),
 Theorem~\ref{thm:toeplitz-symmetric-exact} applies. For general unequal coefficients, however, the alternating-block argument above no longer produces a single local value on every edge, and the exact formula remains open.
\end{remark}

\begin{remark} The failure of the Hilbertian segment picture is already visible
in the ordinary numerical range. Let \(p\ne2\) and let
\(T_0=S_R+S_L\) on \(\ell_p(\mathbb N_0)\), as above. If
\[
        x=(r,is,0,0,\ldots),\qquad r,s>0,\quad r^p+s^p=1,
\]
then a direct computation gives
\[
        \langle T_0x,J(x)\rangle
        =
        i\bigl(sr^{p-1}-rs^{p-1}\bigr).
\]
Thus, whenever \(r\ne s\), this is a nonzero purely imaginary point of
\(W(T_0)\). In particular, for \(p\ne2\), the ordinary numerical range of
\(T_0\) is not contained in the real line.
\end{remark}

\begin{remark}
Even in the simplest symmetric tridiagonal Toeplitz case, the ordinary numerical range is already genuinely two-dimensional on $\ell_p$ when $p\neq2$, although the Hilbertian symbol is the real interval $[-2,2]$. This contrasts with the classical Hardy-space theory, where Toeplitz numerical ranges are convex and, for analytic symbols, are described by the convex hull of the symbol image on $\mathbb D$ \cite{KleinToeplitz}. Thus the discrete Toeplitz example already exhibits the breakdown of the Hilbert-space Toeplitz picture for numerical ranges. Theorem~\ref{thm:toeplitz-symmetric-exact} gives an exact formula in the symmetric case and, by Remark~\ref{rem:toeplitz-balanced-phase-reduction}, in the balanced case $|a_{-1}|=|a_1|$, whereas Theorem~\ref{thm:toeplitz-tridiag-upper} remains only an upper-bound result for general unequal coefficients.
\end{remark}

\subsection{The discrete Hilbert transform on $\ell_p(\mathbb Z)$}

We now turn to bilateral sequence-space operators. Before introducing the
discrete Hilbert transform, we record the bilateral analogue of the shift
argument used above for one-sided Toeplitz operators.

Let $(e_j)_{j\in\mathbb Z}$ be the standard basis of $\ell_p(\mathbb Z)$, and let
$U$ be the bilateral shift given by
\[
        Ue_j=e_{j+1},\qquad j\in\mathbb Z.
\]
Equivalently, $(Ux)_n=x_{n-1}$ for $n\in\mathbb Z$. Again we write $J$ for the duality map on $\ell_p$.

As before, all general results from Section~3 apply here after transporting them through any fixed surjective isometry $\ell_p(\mathbb Z)\cong\ell_p$.

We start with the bilateral analogue of Proposition~\ref{prop:toeplitz-shift-lemma}.
\begin{proposition}\label{prop:bilateral-shift-commuting-ellp}
Let $T\in \mathcal L(\ell_p(\mathbb Z))$ satisfy
\[
TU=UT.
\]
Then
\[
\We(T)=\overline{W(T)}.
\]
\end{proposition}

\begin{proof}
Let $x\in \ell_p(\mathbb Z)$ be a unit vector. Since $T$ commutes with $U$ and $J(U^m x)=U^mJ(x),$
 one has
\[
\ip{T(U^m x)}{J(U^m x)}=\ip{U^mTx}{U^mJ(x)}=\ip{Tx}{J(x)}
\qquad (m\ge 0).
\]
 Moreover, the sequence $(U^m x)$ converges weakly to $0$. Hence every point of $W(T)$ belongs to $\We(T)$, so
\[
W(T)\subset \We(T).
\]
The reverse inclusion
\[
\We(T)\subset \overline{W(T)}
\]
is immediate from the definition of $\We(T)$. Therefore
\[
\We(T)=\overline{W(T)}.
\]
\end{proof}

For finitely supported sequences $a=(a_n)_{n\in\mathbb Z}$, define
\[
(Ha)_n:=\frac1\pi\sum_{m\in\mathbb Z\setminus\{0\}}\frac{a_{n-m}}{m},
\qquad n\in\mathbb Z.
\]
By the discrete M.~Riesz theorem, $H$ extends uniquely to a bounded operator on
$\ell_p(\mathbb Z)$. We denote this extension by $H_{\mathbb Z}$ and call it
the discrete Hilbert transform. 

Equivalently, one may view $H$ as the 
Fourier multiplier on $\ell_p(\mathbb Z)$ acting on finitely supported sequences $(a_n)_{n \in \mathbb Z}$ as 
\[
\widehat{Ha}(\theta)=m(\theta)\,\widehat a(\theta),
\qquad
\widehat a(\theta):=\sum_{n\in\mathbb Z} a_n e^{-in\theta},
\qquad \theta\in\mathbb T,
\]
with symbol
\[
m(\theta)=\frac1\pi\sum_{k\in\mathbb Z\setminus\{0\}}\frac{e^{-ik\theta}}{k}
=i\Bigl(\frac{\theta}{\pi}-1\Bigr),
\qquad 0<\theta<2\pi.
\]

Since $H_{\mathbb Z}$ commutes with the bilateral shift,
Proposition~\ref{prop:bilateral-shift-commuting-ellp} gives
\[
\We(H_{\mathbb Z})=\overline{W(H_{\mathbb Z})}.
\]
If \(p=2\), then \(H_{\mathbb Z}\) is unitarily equivalent to the
multiplication operator on \(L^2(\mathbb T)\) with symbol
\[
        m(\theta)=i\left(\frac{\theta}{\pi}-1\right),
        \qquad 0<\theta<2\pi .
\]
Hence
\[
        \sigma(H_{\mathbb Z})=W_e(H_{\mathbb Z})=i[-1,1],
        \qquad
        W(H_{\mathbb Z})=i(-1,1).
\]
For \(p\ne2\), this one-dimensional picture breaks down: the ordinary and
essential numerical ranges are no longer contained in \(i\mathbb R\), as the
next proposition shows.

\begin{proposition}
\label{prop:discrete-hilbert-l4-disk-small}
Let $H_{\mathbb Z}$ be the Hilbert transform on $\ell_4(\mathbb Z)$. Then
\[
\frac{6}{17\pi}\,\overline\D\subset \We(H_{\mathbb Z})=\overline{W(H_{\mathbb Z})}.
\]
In particular, $\We(H_{\mathbb Z})$ has nonempty interior and is not contained in the imaginary axis.
\end{proposition}

\begin{proof}
Set
\[
x_{\psi}:=\frac{\frac12 e^{i\psi}e_0+e_1}{(1+2^{-4})^{1/4}}
=\frac{\frac12 e^{i\psi}e_0+e_1}{(17/16)^{1/4}}
\qquad (\psi\in\mathbb R).
\]
Since $p=4$, the duality map is given by
\[
J(x)=\bigl(|x_n|^2x_n\bigr)_{n\in\mathbb Z}.
\]
Hence
\[
J(x_{\psi})=\frac{\frac18 e^{i\psi}e_0+e_1}{(17/16)^{3/4}}.
\]
Now $H_{\mathbb Z}x_{\psi}$ is only needed at the coordinates $0$ and $1$:
\[
(H_{\mathbb Z}x_{\psi})_0=-\frac{1}{\pi(17/16)^{1/4}},
\qquad
(H_{\mathbb Z}x_{\psi})_1=\frac{\frac12 e^{i\psi}}{\pi(17/16)^{1/4}}.
\]
Therefore
\[
\ip{H_{\mathbb Z}x_{\psi}}{J(x_{\psi})}
=\frac{-\frac18 e^{-i\psi}+\frac12 e^{i\psi}}{\pi(17/16)}
=\frac{6\cos\psi+10i\sin\psi}{17\pi}.
\]
As $\psi$ varies, this traces the boundary of the ellipse
\[
E:=\left\{\frac{6\cos\psi+10i\sin\psi}{17\pi}:\ \psi\in\mathbb R\right\}.
\]
For every $N\in\mathbb N$, the shifted vector $U^Nx_{\psi}$ has the same numerical value, since $H_{\mathbb Z}$ commutes with $U$. 
Since $U^Nx_{\psi}\wto0$ as $N\to\infty,$ we have
\[
E\subset \We(H_{\mathbb Z}).
\]
By Theorem~\ref{thm:We-ellp-convex}, the set $\We(H_{\mathbb Z})$ is convex, so it contains the convex hull of $E$, that is, the filled ellipse bounded by $E$. This ellipse contains the disc of radius $\frac{6}{17\pi}$ centred at the origin. Hence
\[
\frac{6}{17\pi}\,\overline\D\subset \We(H_{\mathbb Z}).
\]
This proves, in particular, that \(W_e(H_{\mathbb Z})\) is not contained in
\(i\mathbb R\).
\end{proof}

\begin{remark}
The discrete Hilbert transform already shows, within the sequence-space setting of the present paper, that the exact Hilbert-space segment picture is unstable away from $p=2$: on $\ell_4(\mathbb Z)$ the essential numerical range contains a genuine disc.
\end{remark}

\subsection{The unilateral shift and the disc picture}

We conclude the example section with the unilateral shift itself. This returns to the simplest one-sided example and shows how the abstract Banach-space disc theorem and the atomic arguments combine in a particularly transparent form.
The unilateral shift provides the simplest sequence-space example in which the geometry of the numerical range can be described almost completely. We describe this first and then draw the Crouzeix-type consequence.
The proof below could have been based on
Theorem \ref{thm:interior-We-inside-W},
but we preferred a direct argument.
\begin{proposition}\label{cor:shift-disk-lp}
Let \(1<p<\infty\), and let \(S=S_R\) be the unilateral shift on
\(\ell_p(\N_0)\), given by \(Se_j=e_{j+1}\) for the standard basis
\((e_j)_{j\ge0}\). Then
\[
        \We(S)=\overline{\D},\qquad W(S)=\D.
\]
In particular,
\[
        W(S)\cap\T=\varnothing,
        \qquad
        \overline{W(S)}=\We(S)=\overline{\D}.
\]
\end{proposition}
\begin{proof}
Let \(|\lambda|<1\). 
The case \(\lambda=0\) is obtained from \(e_0\). Let now \(0<|\lambda|<1\).
Choose \(\alpha\ne0\), \(|\alpha|<1\), such that
\[
        \overline{\alpha}|\alpha|^{p-2}=\lambda .
\]
Put
\[
        x^{(\alpha)}
        :=
        (1-|\alpha|^p)^{1/p}(1,\alpha,\alpha^2,\ldots).
\]
Then \(\|x^{(\alpha)}\|=1\), and a direct computation gives
\[
        \langle Sx^{(\alpha)},J(x^{(\alpha)})\rangle=\lambda.
\]
Thus \(\mathbb D\subset W(S)\).

For \(m\ge0\), set \(x_m:=S^m x^{(\alpha)}\). Then \(x_m\to0\) weakly and
\[
        \langle Sx_m,J(x_m)\rangle=\lambda.
\]
Hence \(\mathbb D\subset W_e(S)\). Since \(W_e(S)\) is closed and \(S\) is an
isometry, it follows that
\[
        W_e(S)=\overline{\mathbb D}.
\]

It remains to exclude boundary points from \(W(S)\). Let \(x=(x_n)_{n\ge0}\)
be a unit vector. Then
\[
        \langle Sx,J(x)\rangle
        =
        \sum_{n\ge1}x_{n-1}\overline{x_n}|x_n|^{p-2}.
\]
By the triangle inequality and Hölder's inequality,
\[
        |\langle Sx,J(x)\rangle|
        \le
        \left(\sum_{n\ge1}|x_{n-1}|^p\right)^{1/p}
        \left(\sum_{n\ge1}|x_n|^p\right)^{1/q}
        =
        (1-|x_0|^p)^{1/q}
        \le 1.
\]
If equality were attained, then \(x_0=0\), and equality in Hölder's inequality
would give a constant \(c\ge0\) such that
\[
        |x_{n-1}|^p=c|x_n|^p,\qquad n\ge1.
\]
Since both sums in Hölder's inequality are then equal to \(1\), one has
\(c=1\). Thus \(|x_{n-1}|=|x_n|\) for all \(n\ge1\). Together with \(x_0=0\)
this forces \(x=0\), a contradiction. Hence
\[
        |\langle Sx,J(x)\rangle|<1
\]
for every unit vector \(x\), and therefore \(W(S)\subset\mathbb D\). Since
\(\mathbb D\subset W(S)\), we get \(W(S)=\mathbb D\).
\end{proof}

\subsection{The unilateral shift and Crouzeix-type estimates}

We conclude the sequence-space examples by combining the disc behaviour of the unilateral shift with the polynomial estimates behind Crouzeix-type inequalities.  The result is a counterexample on \(\ell_p\).  After recording it, we isolate its finite-dimensional analogue, which leads to the scale
\(
 n^{|1/p-1/2|}.
\)

Recall first the famous Hilbert-space Crouzeix inequality: if $T$ is a Hilbert-space operator, then one has, for every polynomial $f$,
\begin{equation}\label{eq:Crouzeix}
\|f(T)\|\le (1+\sqrt 2)\sup_{z\in \overline{W(T)}} |f(z)|.
\end{equation}
For background on numerical-range spectral-set questions and related functional calculi, see \cite{CrouzeixJFA,CrouzeixPalencia,BBsurvey,GW,CG2019}.  Note that no analogue of \eqref{eq:Crouzeix}, with $1+\sqrt 2$ replaced by an absolute constant $C> 0$, can hold for all operators on $\ell_p$ when $1<p<\infty$ and $p\neq2$.

Indeed, let $S=S_R$ be the unilateral shift on $\ell_p(\mathbb N_0)$.  In view of  Proposition~\ref{cor:shift-disk-lp},
 for every polynomial $f$, we have
\[
\sup_{z\in W(S)}|f(z)|=\sup_{|z|\le1}|f(z)|.
\]
If such an analogue held for \(S\), then the polynomial functional calculus of
\(S\) would be bounded with respect to the disc-algebra norm. We show that this
would imply the same boundedness for the bilateral shift.

Let \(U\) be the bilateral shift on \(\ell_p(\mathbb Z)\), given by \(Ue_j=e_{j+1}\) for the standard basis \((e_j)_{j\in\mathbb Z}\).  Fix an analytic polynomial
\(q\). It is enough to estimate \(q(U)\) on finitely supported vectors. If
\(x\in\ell_p(\mathbb Z)\) has finite support, choose \(N\) so large that \(U^Nx\)
is supported in \(\mathbb N_0\). On this subspace the action of \(U\) agrees with
the action of the unilateral shift \(S\). Hence
\[
q(U)U^Nx=q(S)U^Nx .
\]
Since \(U^N\) is an isometry and commutes with \(q(U)\), we obtain
\[
\|q(U)x\|
=
\|U^Nq(U)x\|
=
\|q(U)U^Nx\|
=
\|q(S)U^Nx\|
\le \|q(S)\| \|x\|.
\]
By density of finitely supported vectors in \(\ell_p(\mathbb Z)\), it follows that
\[
\|q(U)\|\le \|q(S)\|.
\]
Thus polynomial boundedness of \(S\) would imply polynomial boundedness of
\(U\). This contradicts the known fact that the bilateral shift on
\(\ell_p(\mathbb Z)\), \(p\ne2\), is not polynomially bounded, see, for instance,
\cite[Theorem~5.7]{Fixman1959} or \cite[Proposition~4.1]{Cohen2018}.
Therefore the required dimension-free Crouzeix-type estimate cannot hold on
\(\ell_p\).

This proves the following consequence.
\begin{proposition}\label{cor:no-crouzeix-lp}
Let $1<p<\infty$ and $p\neq2$.  Then there is no constant $C\ge1$ such that
\begin{equation}\label{crouz}
\|f(T)\|\le C\sup_{z\in \overline{W(T)}}|f(z)|
\end{equation}
for every polynomial $f$ and every operator $T\in \mathcal L(\ell_p)$.
\end{proposition}
This conclusion agrees with \cite[Theorem~6.1]{BHSVW2025}, formulated there for the algebraic numerical range. Although the algebraic-range version is formally stronger in general, in the shift case considered above \(\overline{W(S)}\) and \(V(S,\mathcal L(\ell_p))\) both coincide with the closed unit disc.

\subsection{Finite-dimensional shift tests and the scale $n^{|1/p-1/2|}$}

The preceding argument only rules out a dimension-free
Crouzeix-type estimate on \(\ell_p\), \(p\ne2\). It leaves open whether, on
\(\ell_p^n\), such an estimate holds with a constant \(C_{p,n}\) depending only
on \(p\) and \(n\), and what the possible growth of \(C_{p,n}\) is as
\(n\to\infty\). However, the constants \(C_{p,n}\), if finite, cannot remain
bounded as \(n\to\infty\): otherwise, applying the corresponding uniform
finite-dimensional estimate to the coordinate compressions \(P_nT|_{\ell_p^n}\),
using \eqref{eq:block-compression} with \(u_j=e_j\), and passing to the strong
limit \(P_nTP_n\to T\), would give the estimate on \(\ell_p\). The finite shifts
considered below point to polynomial asymptotics of order \(n^{\alpha_p}\), where
\[
\alpha_p:=\left|\frac1p-\frac12\right|
       =\max\left\{\frac1p,\frac1q\right\}-\frac12, \qquad q=\frac{p}{p-1}.
\]

If $E\subset\C$ is compact and $f$ is a polynomial, set
\[
\|f\|_E:=\sup_{z\in E}|f(z)|,
\]
and let $e_0,\ldots,e_{n-1}$ be the standard basis of $\ell_p^n$. Let
$S_{n,p} \in \mathcal L(\ell_p^n)$ be the truncated unilateral shift given by
\[
S_{n,p}e_j=e_{j+1}\quad(0\le j\le n-2),
\qquad
S_{n,p} e_{n-1}=0.
\]
Thus $S_{n,p}^n=0$ and $\|S_{n,p}\|=1$.

\begin{proposition}\label{prop:finite-shift-crouzeix-scale}
Let $1<p<\infty$ and $n\ge2$.  Then there is an absolute constant
$c>0$ such that
\[
c\,n^{\alpha_p}
\le
\sup_{f\ne0}\frac{\|f(S_{n,p})\|}{\|f\|_{W(S_{n,p})}}
\le
 e\,n^{\alpha_p},
\]
and
\[
c\,n^{\alpha_p}
\le
\sup_{f\ne0}\frac{\|f(S_{n,p})\|}{\|f\|_{V_{n,p}(S_{n,p})}}
\le
 e\,n^{\alpha_p},
\]
where, for $T\in\mathcal L(\ell_p^n)$,
\[
        V_{n,p}(T):=V\bigl(T,\mathcal L(\ell_p^n)\bigr).
\]
\end{proposition}

\begin{proof}
We first record the elementary disc information.  Fix $\lambda\in\T$ and put
\[
u_{\lambda,n}:=n^{-1/p}\sum_{j=0}^{n-1}\overline\lambda^{\,j}e_j.
\]
Then $\|u_{\lambda,n}\|_p=1$, and a direct computation with the duality map gives
\[
\ip{S_{n,p}u_{\lambda,n}}{J(u_{\lambda,n})}
   =\frac{n-1}{n}\lambda.
\]
Thus, if
\[
r_n:=\frac{n-1}{n},
\]
then
\[
r_n\T\subset W(S_{n,p}),
\qquad
r_n\overline\D\subset V_{n,p}(S_{n,p}).
\]
Since $S_{n,p}$ is a contraction, we also have
\[
W(S_{n,p})\subset\overline\D,
\qquad
V_{n,p}(S_{n,p})\subset\overline\D.
\]

For the lower estimate we use flat polynomials, as in
\cite[Theorem~4.3]{BHSVW2025}. Here the classical Rudin--Shapiro construction
suffices: there is an absolute constant $C_{\rm RS}$ such that, whenever
$m=2^s$, one can choose signs
\[
        \varepsilon_0,\ldots,\varepsilon_{m-1}\in\{-1,1\}
\]
so that
\[
        f_m(z)=\sum_{k=0}^{m-1}\varepsilon_k z^k
\]
satisfies
\[
        \|f_m\|_{\overline\D}\le C_{\rm RS}\sqrt m .
\]
See, e.g., \cite{Rudin1959}.

Fix $n\ge2$, and choose $m=2^s$ such that $n/2<m\le n$. We regard $f_m$
as a polynomial of degree at most $n-1$, with zero coefficients in the remaining
positions. Since
\[
f_m(S_{n,p})e_0=\sum_{k=0}^{m-1}\varepsilon_k e_k,
\]
we get
\[
\|f_m(S_{n,p})\|\ge m^{1/p}.
\]
Applying the same argument to the adjoint gives the complementary estimate.
Indeed, $S_{n,p}^*$ is the backward shift on $\ell_q^n$, and
\[
        f_m(S_{n,p})^* e_{m-1}
        =
        \sum_{k=0}^{m-1}\varepsilon_k e_{m-1-k}.
\]
This vector has $\ell_q^n$-norm $m^{1/q}$. Hence
\[
        \|f_m(S_{n,p})\|
        =
        \|f_m(S_{n,p})^*\|
        \ge m^{1/q}.
\]
Therefore
\[
\|f_m(S_{n,p})\|\ge m^{\max\{1/p,1/q\}}.
\]
Because $W(S_{n,p})\subset V_{n,p}(S_{n,p})\subset\overline\D$, it follows that
\[
\frac{\|f_m(S_{n,p})\|}{\|f_m\|_{W(S_{n,p})}}
\ge
\frac{\|f_m(S_{n,p})\|}{\|f_m\|_{V_{n,p}(S_{n,p})}}
\ge
\frac{\|f_m(S_{n,p})\|}{\|f_m\|_{\overline\D}}
\ge
C_{\rm RS}^{-1}m^{\alpha_p}.
\]
Since $m>n/2$ and $0\le\alpha_p\le1/2$, the last quantity is bounded below
by $c n^{\alpha_p}$ with an absolute constant $c>0$.

We turn to the upper estimate.  Note that for every operator $T\in \mathcal L(\ell_p^n)$,
\begin{equation}\label{eq:p-2-comparison-matrix}
\|T\|_{\mathcal L(\ell_p^n)}\le n^{\alpha_p}\|T\|_{\mathcal L(\ell_2^n)}.
\end{equation}
This follows immediately from the standard norm comparisons between $\ell_p^n$ and $\ell_2^n$.

Let $r=r_n=(n-1)/n$ and define the diagonal operator $D_{r,2}$ on $\ell_2^n$ by
\[
D_{r,2}e_j=r^{-j}e_j,
\qquad 0\le j\le n-1.
\]
Then
\[
D_{r,2}^{-1}S_{n,2}D_{r,2}=rS_{n,2}.
\]
Hence, for every polynomial $f$,
\[
f(S_{n,2})=D_{r,2} f(rS_{n,2})D_{r,2}^{-1}.
\]
Since $S_{n,2}$ is a Hilbert-space contraction, von Neumann's inequality applied to $z\mapsto f(rz)$ gives
\[
\|f(rS_{n,2})\|\le \|f\|_{r\overline\D}.
\]
Moreover,
\[
\|D_{r,2}\| \|D_{r,2}^{-1}\|=r^{-(n-1)}
       =\left(\frac{n}{n-1}\right)^{n-1}\le e.
\]
Therefore
\[
\|f(S_{n,2})\| \le e\,\|f\|_{r\overline\D}.
\]
Combining this with \eqref{eq:p-2-comparison-matrix}, we obtain
\[
\|f(S_{n,p})\|\le e\,n^{\alpha_p}\|f\|_{r\overline\D}.
\]
Finally, $r\overline\D\subset V_{n,p}(S_{n,p})$, while $r\T\subset W(S_{n,p})$ and the maximum principle gives
\[
\|f\|_{r\overline\D}=\sup_{|z|=r}|f(z)|\le \|f\|_{W(S_{n,p})}
\le \|f\|_{V_{n,p}(S_{n,p})}.
\]
This proves both upper estimates.
\end{proof}

The same exponent is not specific to truncated shifts.  It also appears for cyclic shifts.
Let $e_0,\ldots,e_{n-1}$ be the standard basis of $\ell_p^n$, and let
$U_{n,p}\in\mathcal L(\ell_p^n)$ be the cyclic shift given by
\[
        U_{n,p}e_j=e_{j+1}\quad(0\le j\le n-2),\qquad
        U_{n,p}e_{n-1}=e_0.
\]
\begin{proposition}\label{prop:cyclic-shift-crouzeix-scale}
Let $1<p<\infty$ and $n\ge2$. Then there is an absolute constant $c>0$ such that
\[
c\,n^{\alpha_p}
\le
\sup_{f\ne0}\frac{\|f(U_{n,p})\|}{\|f\|_{W(U_{n,p})}}
\le
 n^{\alpha_p},
\]
and
\[
c\,n^{\alpha_p}
\le
\sup_{f\ne0}\frac{\|f(U_{n,p})\|}{\|f\|_{V_{n,p}(U_{n,p})}}
\le
 n^{\alpha_p},
\]
where
\[
        V_{n,p}(U_{n,p}):=V\bigl(U_{n,p},\mathcal L(\ell_p^n)\bigr).
\]
\end{proposition}

\begin{proof}
For the lower bound, choose again $m=2^s$ with $n/2<m\le n$ and use
the polynomial $f_m$ from the proof of
Proposition~\ref{prop:finite-shift-crouzeix-scale}; in particular,
$\|f_m\|_{\overline\D}\le C_{\rm RS}\sqrt m$. Since $m\le n$, the
vectors $e_0,U_{n,p}e_0,\dots,U_{n,p}^{m-1}e_0$ are distinct standard basis
vectors. Hence
\[
\|f_m(U_{n,p})e_0\|=m^{1/p},
\]
and the adjoint argument gives the complementary lower bound $m^{1/q}$. Thus
\[
        \|f_m(U_{n,p})\|\ge m^{\max\{1/p,1/q\}}.
\]
As $U_{n,p}$ is an isometry on $\ell_p^n$, both $W(U_{n,p})$ and
$V_{n,p}(U_{n,p})$ are contained in $\overline\D$. Hence
\[
\frac{\|f_m(U_{n,p})\|}{\|f_m\|_{W(U_{n,p})}}
\ge
\frac{\|f_m(U_{n,p})\|}{\|f_m\|_{\overline\D}}
\ge
C_{\rm RS}^{-1}m^{\alpha_p}
\ge c n^{\alpha_p},
\]
and the same estimate holds with $W(U_{n,p})$ replaced by
$V_{n,p}(U_{n,p})$.

For the upper bound, use \eqref{eq:p-2-comparison-matrix}.  Since $U_{n,2}$ is unitary on $\ell_2^n$,
\[
\|f(U_{n,2})\|=\max_{\omega^n=1}|f(\omega)|.
\]
The points $\omega$ with $\omega^n=1$ are eigenvalues of $U_{n,p}$ and are realised by their unimodular eigenvectors, hence they belong to $W(U_{n,p})$. Therefore
\[
\|f(U_{n,p})\|
\le n^{\alpha_p}\max_{\omega^n=1}|f(\omega)|
\le n^{\alpha_p}\|f\|_{W(U_{n,p})}
\le n^{\alpha_p}\|f\|_{V_{n,p}(U_{n,p})}.
\]
This proves the asserted estimates.
\end{proof}

We may now formulate the finite-dimensional Crouzeix-type inequality in the $\ell_p$ setting, which looks natural in view of the above observations.  Define
\[
\mathcal C^W_{n,p}:=
\sup_{T\in \mathcal L(\ell_p^n)}\sup_{f\ne0}
\frac{\|f(T)\|}{\|f\|_{W(T)}},
\]
and
\[
\mathcal C^V_{n,p}:=
\sup_{T\in \mathcal L(\ell_p^n)}\sup_{f\ne0}
\frac{\|f(T)\|}{\|f\|_{V(T)}}.
\]
Since $W(T)\subset V(T)$, one has
\[
\mathcal C^V_{n,p}\le \mathcal C^W_{n,p}.
\]
Proposition~\ref{prop:finite-shift-crouzeix-scale} gives
\[
\mathcal C^V_{n,p}\ge c n^{\alpha_p},
\qquad
\mathcal C^W_{n,p}\ge c n^{\alpha_p}
\]
with an absolute constant $c>0$.
Thus we arrive at the next curious question.
\begin{question}\label{q:finite-dimensional-crouzeix-growth}
For fixed $1<p<\infty$, is there a constant $C_p>0$ such that
\[
\mathcal C^V_{n,p}\le C_p n^{\alpha_p}\qquad(n\in\mathbb N)?
\]
More strongly, is there a constant $C'_p>0$ such that
\[
\mathcal C^W_{n,p}\le C'_p n^{\alpha_p}\qquad(n\in\mathbb N)?
\]
\end{question}

The $V$-version is algebraically more natural, but $V(T)$ is difficult to determine explicitly.  The $W$-version is closer to the original Crouzeix formulation, but it is technically harder because $W(T)$ is typically non-convex.
There are several natural ways to test Question~\ref{q:finite-dimensional-crouzeix-growth}, but we omit their discussion.

The examples discussed in this section complete the atomic example part of the paper. They illustrate that the general theory from Sections~\ref{inflp} and~\ref{secexamp} is not only formal: even for very concrete sequence-space operators, the essential numerical range on $\ell_p$ can already display geometry with no Hilbert-space analogue, and the unilateral shift shows that this geometry does not support a Hilbert-space-type Crouzeix theory on $\ell_p$. Nevertheless, the finite-dimensional shift examples suggest that some vestiges of such a theory may remain at the finite-dimensional level.

\section{Possible extensions}\label{secexten}

\subsection{Spaces of class $(P)$}

We close the paper by indicating two directions in which the preceding arguments extend beyond the atomic $\ell_p$ setting. The first one concerns Banach spaces of class $(P)$.

The methods developed above are not tied exclusively to $\ell_p$.
They extend, with only minor changes in notation, to Banach spaces carrying a compatible
exhaustion by finite-rank projections. We say that a Banach space $X$ belongs to class $(P)$ if there exists a sequence $(P_n)_{n\ge1}$
of non-zero finite-rank projections such that:
\begin{enumerate}[label=\textup{(\roman*)}]
\item $P_mP_n=P_nP_m=P_{\min\{m,n\}}$ for all $m,n\in\mathbb N$, and $P_nx\to x$ for every $x\in X$;
\item for every \(n\in\mathbb N\), whenever \(x,y\in X\) satisfy
\[
        \|P_nx\|\le \|P_ny\|,
        \qquad
        \|(I-P_n)x\|\le \|(I-P_n)y\|,
\]
one has \(\|x\|\le \|y\|\);
\item for every $x^*\in X^*$ one has
\[
        \bigl\|x^*|_{(I-P_n)X}\bigr\|\to0
        \qquad (n\to\infty),
\]
or, equivalently, $\|(I-P_n)^*x^*\|\to0$.
\end{enumerate}

If \(X\) is of class \((P)\), then \(P_nX\) is finite-dimensional and
\((I-P_n)X\) is finite-codimensional for every \(n\). Moreover,
condition \textup{(ii)} implies
\[
        \|P_n\|=1,\qquad \|I-P_n\|=1,
        \qquad n\in\mathbb N,
\]
provided the corresponding projection is non-zero. Indeed, apply
\textup{(ii)} first to \(P_nx\) and \(x\), and then to \((I-P_n)x\) and \(x\).
Also, \(X\) is separable, because \(\bigcup_n P_nX\) is dense in \(X\), while
\(X^*\) is separable by \textup{(iii)}. Thus class \((P)\) isolates spaces with
finite-dimensional initial parts, finite-codimensional tails and a compatible
dual tail approximation.

If, in addition,
\[
        \dim (P_n-P_{n-1})X=1,\qquad n\ge1,\quad P_0:=0,
\]
then, after choosing normalised vectors
\[
        e_n\in (P_n-P_{n-1})X,\qquad n\ge1,
\]
the sequence \((e_n)_{n\ge1}\) is a shrinking \(1\)-unconditional basis of
\(X\). Thus class \((P)\) may be viewed as a finite-dimensional-decomposition
version of the shrinking \(1\)-unconditional basis setting.

Thus, in the setting of Banach spaces of class \((P)\), one again has finite-dimensional initial blocks and
finite-codimensional tails, and the arguments based on block decomposition,
tail approximation, and weak convergence go through with only minor changes
in the proofs. In particular, the description of the essential numerical range, its convexity,
and the star-centre mechanism admit analogues for spaces of class \((P)\).

Besides \(\ell_p\), natural examples are provided by \(\ell_p\)- or \(c_0\)-sums of
a sequence of finite-dimensional Banach spaces $(E_n)$:
\[
        \Big(\bigoplus_{n=1}^\infty E_n\Big)_{\ell_p},
        \qquad 1<p<\infty,
        \qquad\text{and}\qquad
        \Big(\bigoplus_{n=1}^\infty E_n\Big)_{c_0}.
\]
In the \(c_0\)-case the numerical range is understood in the general Banach-space
sense of \eqref{eq:banach-spatial-numerical-range}, not through a duality map.
We do not pursue this extension here. The point of the formulation above is
rather to indicate that the present arguments depend only on a small number
of structural features of \(\ell_p\).
A related way of isolating sequence-space arguments through assumptions on a
decomposition of the underlying Banach space appears in \cite{H}, where
\(c_0\) and \(\ell_p\), \(1<p<\infty\), are obtained as applications of a more
general theorem.

\subsection{Joint essential numerical ranges}

A second extension concerns several operators simultaneously. The same ideas lead naturally to several new properties of the joint essential numerical range on $\ell_p$-spaces. Let $1<p<\infty$. For an $r$-tuple $(T_1,\dots,T_r)\in \mathcal L(\ell_p)^r$, define the joint numerical range by
\[
W(T_1,\dots,T_r):=
\bigl\{(\langle T_1x,J(x)\rangle,\dots,\langle T_rx,J(x)\rangle):
x\in\ell_p,\ \|x\|=1\bigr\}\subset\mathbb C^r.
\]
The joint essential numerical range \(W_e(T_1,\dots,T_r)\) is defined as the set of all $(\lambda_1,\dots,\lambda_r)\in\C^r$ such that there exists a weakly-null sequence 
$(x_n)\subset\ell_p$ of unit vectors with
\[
\langle T_j x_n,J(x_n)\rangle\to \lambda_j
\qquad (j=1,\dots,r).
\]

The arguments of the preceding section extend almost literally to finite tuples, where, in particular, the Calkin-algebra formulation is understood in the usual joint algebraic numerical range sense.
We record the resulting formulation and omit the proof. 
Here the symbol $V$ in the quotient-algebra formula denotes the joint algebraic numerical range.

\begin{theorem} 
Let $(T_1,\dots,T_r)\in \mathcal L(\ell_p)^r$. Then:
\begin{itemize}
\item[(i)] $W_e(T_1,\dots,T_r)$ is a nonempty compact convex subset of $\C^r$;
\item[(ii)] $W_e(T_1,\dots,T_r)=V\bigl(\pi(T_1),\dots,\pi(T_r),\mathcal L(\ell_p)/\mathcal K(\ell_p)\bigr)$;
\item[(iii)] $W_e(T_1,\dots,T_r)=\bigcap\{\overline{W(T_1+L_1,\dots,T_r+L_r)}: L_1,\dots,L_r\in \mathcal K(\ell_p)\}$;
\item[(iv)] $\overline{W(T_1,\dots,T_r)}$ is star-shaped, and every point of $W_e(T_1,\dots,T_r)$ is a star-centre of this set.
\item[(v)] $\operatorname{Int} W_e(T_1,\dots,T_r)\subset W(T_1,\dots,T_r)$. If this interior is nonempty, then $W(T_1,\dots,T_r)$ is star-shaped, with every point of the interior as a star-centre.
\end{itemize}
\end{theorem}

Hilbert-space counterparts of the tuple statements above can be found in
\cite{LiPoon2009, Muller2010, MT-circles, MT-jointnr}. See also \cite[Chapter~8.1]{GW}. 
\appendix

\section{A computational proof of the $\ell_3^2$ counterexample}\label{app:lp2-counterexample-proof}

We supply here the full computation for Proposition~\ref{prop:lp2-counterexample-full}.

\begin{proof}[Proof of Proposition~\ref{prop:lp2-counterexample-full}]
For $p=3$ and $q=3/2$, write
\[
r=\frac{t^3}{1+t^3},\qquad t>0,
\]
and introduce the shorthand
\[
s_t:=2r-1=\frac{t^3-1}{1+t^3}.
\]
Then
\[
\beta(r)=\frac{t^2}{1+t^3},
\qquad
\gamma(r)=\frac{t}{1+t^3}.
\]
Applying Proposition~\ref{prop:lp2-reduction} to the matrix $S$, one obtains ellipses
\[
E_t=
\left\{
 s_t\!\left(8+\frac{i}{5}\right)+a_t\cos\eta+i\,\varepsilon_t b_t\sin\eta:
 \eta\in\R
\right\},
\]
where $\varepsilon_t\in\{\pm1\}$ is irrelevant for the underlying ellipse and
\[
a_t=\frac{t^2+\frac65 t}{1+t^3}=\frac{5t^2+6t}{5(1+t^3)},
\qquad
b_t=\frac{\bigl|t^2-\frac65 t\bigr|}{1+t^3}=\frac{|5t^2-6t|}{5(1+t^3)}.
\]
Thus
\[
W(S)=\bigcup_{t>0} E_t.
\]

We first determine which parameters can contribute points on the positive real ray.
If $\rho\in E_t\cap[0,\infty)$, then comparing imaginary parts gives
\[
0=\frac{s_t}{5}+\varepsilon_t b_t\sin\eta,
\]
hence necessarily
\[
\left|\frac{s_t}{5}\right|\le b_t.
\]
Equivalently,
\begin{equation}\label{eq:counter-admissibility-paper1}
|t^3-1|\le |5t^2-6t|.
\end{equation}
If $1\le t\le 6/5$, then \eqref{eq:counter-admissibility-paper1} becomes
\[
Q(t):=t^3+5t^2-6t-1\le 0.
\]
Note that
$Q$ is strictly increasing on $[1,\infty)$, and
\[
Q\!\left(\frac98\right)=\frac1{512}>0.
\]
Hence no $t\in[9/8,6/5]$ satisfies \eqref{eq:counter-admissibility-paper1}.
If $t\ge 6/5$, then \eqref{eq:counter-admissibility-paper1} becomes
\[
P(t):=t^3-5t^2+6t-1\le 0.
\]
Observe that
$P$ is strictly decreasing on $[6/5,31/20]$, while
\[
P\!\left(\frac{31}{20}\right)=\frac{91}{8000}>0.
\]
Therefore no $t\in[6/5,31/20]$ satisfies \eqref{eq:counter-admissibility-paper1}. Consequently, if $E_t$ meets the positive real ray, then necessarily
\begin{equation}\label{eq:counter-gap-in-t-paper1}
t<\frac98
\qquad\text{or}\qquad
t>\frac{31}{20}.
\end{equation}

Next we obtain a uniform bound for the horizontal semiaxis.
For all $t>0$ one has
\[
a_t\le \frac65.
\]
Indeed,
\[
\frac65-a_t=
\frac{6t^3-5t^2-6t+6}{5(1+t^3)}.
\]
This positivity can be checked without using the discriminant. Indeed, with
$h(t):=6t^3-5t^2-6t+6$, one has the elementary decomposition
\[
        h(t)=\frac{(6t-5)^2(12t+5)}{108}
             +2(t-1)^2(t+2)+\frac{91}{108}.
\]
All terms on the right-hand side are non-negative for $t>0$, and the last one
is strictly positive. Hence $h(t)>0$ for all $t>0$.

We now show that the point \(3\) is missing from \(W(S)\).
If $\rho\in E_t\cap[0,\infty)$, then
\[
\rho=8s_t+a_t\cos\eta,
\]
hence
\begin{equation}\label{eq:counter-basic-band-paper1}
8s_t-a_t\le \rho\le 8s_t+a_t.
\end{equation}
If $t<9/8$, then by monotonicity of $s_t$, \eqref{eq:counter-basic-band-paper1}, and the uniform bound on the horizontal semiaxis obtained above,
\[
\rho\le 8s_{9/8}+\frac65=\frac{16126}{6205}<3.
\]
If $t>31/20$, then again by monotonicity of $s_t$,
\[
\rho\ge 8s_{31/20}-\frac65=\frac{644894}{188955}>3.
\]
Together with \eqref{eq:counter-gap-in-t-paper1}, this shows that no positive real point of $W(S)$ can equal $3$. Hence $3\notin W(S)$.

It remains to prove that the points $0$ and $6$ do belong to $W(S)$.
Since the diagonal entries of $S$ are $8+i/5$ and $-8-i/5$, Corollary~\ref{cor:lp2-diagonal-segment} gives
\[
0=\frac12\Bigl(8+\frac{i}{5}\Bigr)+\frac12\Bigl(-8-\frac{i}{5}\Bigr)\in W(S).
\]
At $t=2$ one has
\[
s_2=\frac79,
\qquad
a_2=\frac{32}{45},
\qquad
b_2=\frac{8}{45},
\]
so
\[
E_2=
\left\{
\frac{56}{9}+\frac{7i}{45}+\frac{32}{45}\cos\eta+\frac{8i}{45}\sin\eta:
\eta\in\R
\right\}.
\]
Its intersection with the real axis is obtained from
\[
\frac{7}{45}+\frac{8}{45}\sin\eta=0
\quad\Longleftrightarrow\quad
\sin\eta=-\frac78,
\]
hence
\[
E_2\cap\R=
\left[
\frac{56}{9}-\frac{4\sqrt{15}}{45},
\,\frac{56}{9}+\frac{4\sqrt{15}}{45}
\right].
\]
Since $4\sqrt{15}>10$, one has
\[
\frac{56}{9}-\frac{4\sqrt{15}}{45}<6<\frac{56}{9}+\frac{4\sqrt{15}}{45},
\]
so $6\in E_2\subset W(S)$.

Thus, we have proved that $0$ and $6$ belong to $W(S)$ while $3\notin W(S),$ so that $W(S)$ is not star-shaped with centre $0$.
\end{proof}


\begin{thebibliography}{99}

\bibitem{AglerLykovaYoung2024}
J.~Agler, Z.~A.~Lykova, and N.~J.~Young,
\emph{On the operators with numerical range in an ellipse},
J. Funct. Anal. \textbf{287} (2024),  Paper No.~110556, 62 pp.



\bibitem{AllenWard79}
G.~D.~Allen and J.~D.~Ward,
\emph{Hermitian liftings in $B(\ell_p)$},
J. Operator Theory \textbf{1} (1979), 27--36.

\bibitem{AppDV}
J.~Appell, E.~De~Pascale, and A.~Vignoli,
\emph{Nonlinear Spectral Theory},
De Gruyter, Berlin, 2004.

\bibitem{BBsurvey}
C.~Badea and B.~Beckermann,
\emph{Spectral sets},
in \emph{Handbook of Linear Algebra}, 2nd ed., L.~Hogben, ed.,
CRC Press, Boca Raton, FL, Chapter 37, 2014.


 
\bibitem{BanasMursaleen}
J.~Bana\'s and M.~Mursaleen,
\emph{Sequence Spaces and Measures of Noncompactness with Applications to Differential and Integral Equations},
Springer, New Delhi, 2014.

\bibitem{BM}
M.~Barraa and V.~M\"uller,
\emph{On the essential numerical range},
Acta Sci. Math. (Szeged) \textbf{71} (2005), 285--298.

\bibitem{Bickel2020}
K. Bickel, P. Gorkin, A. Greenbaum, T. Ransford, F. Schwenninger, and E. Wegert, \emph{Crouzeix's conjecture and related problems,} Comput. Methods Funct. Theory \textbf{20} (2020), 701--728. 


\bibitem{BHSVW2025}
H.~Blazhko, D.~Homza, F.~L.~Schwenninger, J.~de~Vries, and M.~Wojtylak,
\emph{The algebraic numerical range as a spectral set in Banach algebras},
Canad. J. Math. (2025), \texttt{ https://doi.org/10.4153/S0008414X25000124}, published online.



\bibitem{Blecher2019} D. P. Blecher and N. C. Phillips, \emph{
$L^p$-operator algebras with approximate identities,} Pacific J. Math. \textbf{303} (2019), 401--457.


\bibitem{Boedihardjo2024}
M.~T. Boedihardjo, \emph{
Embedding \(C^*\)-algebras into the Calkin algebra of \(\ell_p\),}
J. Funct. Anal. \textbf{287} (2024),  Paper No.~110669.

\bibitem{Bogli2020} S. B\" ogli, M. Marletta, and C. Tretter, \emph{The essential numerical range for unbounded linear operators,} J. Funct. Anal. \textbf{279} (2020), Paper No.~108509, 49 pp.


\bibitem{BD1}
E.~F.~Bonsall and J.~Duncan,
\emph{Numerical Ranges of Operators on Normed Spaces and of Elements of Normed Algebras},
London Math. Soc. Lecture Note Series 2, Cambridge Univ. Press, London, 1971.

\bibitem{BD2}
E.~F.~Bonsall and J.~Duncan,
\emph{Numerical Ranges}, II,
London Math. Soc. Lecture Note Series 10, Cambridge Univ. Press, London, 1973.



\bibitem{Cio}
I.~Cior\u{a}nescu,
\emph{Geometry of Banach Spaces, Duality Mappings and Nonlinear Problems},
Mathematics and its Applications, 62, Kluwer Academic Publishers Group, Dordrecht, 1990.

\bibitem{ChSSW}
C.~K.~Chui, P.~W.~Smith, R.~R.~Smith, and J.~D.~Ward,
\emph{$L$-ideals and numerical range preservation},
Illinois J. Math. \textbf{21} (1977), 365--373.


\bibitem{Cohen2018}
G. Cohen, \emph{Doubly power bounded operators on $L^p, 2 \neq p>1,$}
 J. Math. Anal. Appl. \textbf{466} (2018), 1327--1336.
 
 
 \bibitem{CrouzeixJFA}
M.~Crouzeix,
\emph{Numerical range and functional calculus in Hilbert space},
J. Funct. Anal. \textbf{244} (2007),  668--690.

\bibitem{CrouzeixPalencia}
M.~Crouzeix and C.~Palencia,
\emph{The numerical range is a $(1+\sqrt2)$-spectral set},
SIAM J. Matrix Anal. Appl. \textbf{38} (2017),  649--655.


\bibitem{CG2019}
M.~Crouzeix and A.~Greenbaum,
\emph{Spectral sets: numerical range and beyond},
SIAM J. Matrix Anal. Appl. \textbf{40} (2019), 1087--1101.



\bibitem{Duren1964}
P.~L.~Duren,
\emph{On the spectrum of a Toeplitz operator},
Pacific J. Math. \textbf{14} (1964), 21--29.


\bibitem{FSW}
P.~A.~Fillmore, J.~G.~Stampfli, and J.~P.~Williams,
\emph{On the essential numerical range, the essential spectrum, and a problem of Halmos},
Acta Sci. Math. (Szeged) \textbf{33} (1972), 179--192.


\bibitem{Fixman1959}
U.~Fixman,
\emph{Problems in spectral operators},
Pacific J. Math. \textbf{9} (1959), 1029--1051.


\bibitem{GW}
H.-L.~Gau and P.~Y.~Wu,
\emph{Numerical Ranges of Hilbert Space Operators},
Encyclopedia of Mathematics and its Applications~179, Cambridge Univ. Press, Cambridge, 2021.


\bibitem{GR}
K.~E.~Gustafson and D.~K.~M.~Rao,
\emph{Numerical Range. The Field of Values of Linear Operators and Matrices},
Universitext, Springer, New York, 1997.


\bibitem{H}
J.~Hennefeld,
\emph{A decomposition for $B(X)^*$ and unique Hahn--Banach extensions},
Pacific J. Math. \textbf{46} (1973), 197--199.


\bibitem{KleinToeplitz}
E.~M.~Klein,
\emph{The numerical range of a Toeplitz operator},
Proc. Amer. Math. Soc. \textbf{35} (1972), 101--103.



\bibitem{LauLiPoonSze2018}
P.-S.~Lau, C.-K.~Li, Y.-T.~Poon, and N.-S.~Sze,
\emph{Convexity and star-shapedness of matricial range},
J. Funct. Anal. \textbf{275} (2018),  2497--2515.


\bibitem{LaursenNeumann}
K.~B.~Laursen and M.~M.~Neumann,
\emph{An Introduction to Local Spectral Theory},
London Mathematical Society Monographs, New Series, vol.~20,
Clarendon Press, Oxford, 2000.

\bibitem{LiPoon2009}
C.-K.~Li and Y.-T.~Poon,
\emph{The joint essential numerical range of operators: convexity and related results},
Studia Math. \textbf{194} (2009), 91--104.

\bibitem{Malman2025} B. Malman, J. Mashreghi, R. O’Loughlin, and T. Ransford, \emph{Double-layer potentials, configuration constants, and applications to numerical ranges,} Intern. Math. Res. Notices, no. 8, article ID rnaf084, 34 pp. (2025).

\bibitem{Mandal}
 K. Mandal, A. Bhanja, S. Bag, and K. Paul, \emph{On the numerical range of operators on some special Banach spaces,}
  J. Convex Anal. \textbf{29} (2022),  371--380.

\bibitem{Muller2010}
V.~M\"uller,
\emph{The joint essential numerical range, compact perturbations, and the Olsen problem},
Studia Math. \textbf{197} (2010), 275--290.

\bibitem{MullerBook}
V.~M\"uller,
\emph{Spectral Theory of Linear Operators and Spectral Systems in Banach Algebras},
2nd ed., Operator Theory: Advances and Applications~139, Birkh\"auser, Basel, 2007.

\bibitem{MT-circles}
V.~M\"uller and Yu.~Tomilov,
\emph{Circles in the spectrum and the geometry of orbits: A numerical ranges approach},
J. Funct. Anal. \textbf{274} (2018),  433--460.


\bibitem{MT-jointnr}
V.~M\"uller and Yu.~Tomilov,
\emph{Joint numerical ranges and compressions of powers of operators},
J. Lond. Math. Soc. (2) \textbf{99} (2019),  127--152.

\bibitem{MT}
V.~M\"uller and Yu.~Tomilov,
\emph{In search of convexity: diagonals and numerical ranges},
Bull. Lond. Math. Soc. \textbf{53} (2021),  1016--1029.

\bibitem{MTnonatomic}
V.~M\"uller and Yu.~Tomilov,
\emph{Numerical and essential numerical ranges on non-atomic $L^p$-spaces},
preprint.

\bibitem{RubinWesler1958}
H.~Rubin and O.~Wesler,
\emph{A note on convexity in Euclidean $n$-space},
Proc. Amer. Math. Soc. \textbf{9} (1958), 522--523.

\bibitem{Rudin1959}
W.~Rudin,
\emph{Some theorems on Fourier coefficients},
Proc. Amer. Math. Soc. \textbf{10} (1959), 855--859.

\bibitem{vanNeervenFA}
J.~van Neerven,
\emph{Functional Analysis},
Cambridge Studies in Advanced Mathematics, vol.~\textbf{201},
Cambridge Univ. Press, Cambridge, 2022.

\bibitem{Wang2021}
Q.~Wang and Z.~Wang,
\emph{Notes on the \(\ell_p\)-Toeplitz algebra on \(\ell_p(\mathbb N)\)},
Israel J. Math. \textbf{245} (2021), 153--163.

\end{thebibliography}
\end{document}